\documentclass[11pt,eqright]{amsart}
\usepackage{amssymb,amsmath}
\usepackage{bm}
\usepackage{amsfonts}
\usepackage{epsf}
\usepackage{graphicx}
\newtheorem{theorem}{Theorem}[section]
\newtheorem{lemma}[theorem]{Lemma}

\newtheorem{remark}[theorem]{Remark}

\theoremstyle{definition}

\numberwithin{equation}{section}

\newcommand{\bc}{\begin{center}}
\newcommand{\ec}{\end{center}}
\newcommand{\be}{\begin{eqnarray}}
\newcommand{\ee}{\end{eqnarray}}
\newcommand{\nn}{\nonumber}
\newcommand{\ben}{\begin{eqnarray*}}
\newcommand{\een}{\end{eqnarray*}}

\newcommand{\Om}{\Omega}

\newcommand{\lam}{\lambda}

\newcommand{\pa}{\partial}

\newcommand{\na}{\nabla}

\def\x{\times}
\def\hK{\hat{K}}
\def\na{\nabla}

\def\cC{\mathcal{C}}

\def\cT{\mathcal{T}}
\def\R{\mathbb{R}}

\def\bpsi{\boldsymbol{\psi}}

\def\div{\operatorname{div}}
\def\rot{\operatorname{rot}}

\DeclareMathOperator{\sspan}{span}

\DeclareMathOperator{\dif}{d}

\newcommand\diff{\,\dif}

\title[]
{\small  Lower Bounds for Eigenvalues of  Elliptic Operators\\---
By Nonconforming Finite Element Methods}
\author[J.~Hu,  Y.~Huang, and  Q.~Lin]
{Jun Hu$^\ast$, Yunqing Huang$^\dagger$, and Qun Lin$^\ddag$}
\address{$^\ast$ LMAM and School of Mathematical Sciences,
 Peking University, Beijing 100871, P. R. China}
\email{hujun@math.pku.edu.cn}
\address{$^\dagger$ Hunan Key Laboratory for Computation and
Simulation in Science and Engineering, Xiangtan
 University, Xiangtan 411105, P.R.China}
\email{huangyq@xtu.edu.cn}

\address{$^\ddag$ LSEC, ICMSEC, Academy of Mathematics and Systems
Science, Chinese Academy of Sciences, Beijing 100190, China.}
\email{linq@lsec.cc.ac.cn}

\thanks{The first author was supported by  the NSFC Project 11271035; the second
author was supported by NSFC the Key Project 11031006, IRT1179 of PCSIRT and 2010DFR00700.}

\keywords{Lower bound, nonconforming element, eigenvalue,  elliptic
operator.
\\ AMS Subject Classification: 65N30,  65N15, 35J25}
\begin{document}

\newpage
\begin{abstract} The aim of the paper is to introduce a new systematic method
that can produce lower bounds for eigenvalues.
The main idea is to use nonconforming finite element methods.  The general
conclusion herein is that if  local approximation properties of
nonconforming finite element spaces $V_h$ are better than global
continuity properties of $V_h$,  corresponding methods  will produce  lower bounds for  eigenvalues.
More precisely, under three conditions  on  continuity and
approximation properties of nonconforming finite element spaces
we first show abstract error estimates of approximate
eigenvalues and eigenfunctions. Subsequently, we propose one more condition
and prove that it is sufficient to guarantee  nonconforming
finite element methods to produce lower bounds for eigenvalues
of symmetric elliptic operators. As one application,  we
 show that this condition hold for most  nonconforming elements in literature. As another
important application, this condition provides a guidance to  modify
known nonconforming elements in  literature and to propose
new nonconforming elements.  In fact,  we enrich locally the
Crouzeix-Raviart element such that the new element satisfies the
condition;  we  propose a new nonconforming element for  second
order elliptic operators and prove that it will yield lower
bounds for eigenvalues. Finally, we prove the saturation condition
for most  nonconforming elements.
\end{abstract}
\maketitle

\section{Introduction}\label{S1}
Finding eigenvalues of partial differential operators is important
in the mathematical science. Since exact eigenvalues are almost
impossible, many papers and books investigate their bounds from
above and below. It is well known that the variational principle
(including  conforming finite element methods) provides upper
bounds. But the problem of obtaining lower bounds is generally considerably more difficult.
Moreover, a simple combination of lower and upper bounds will
produce intervals to which  exact eigenvalue belongs. This in
turn gives  reliable a posteriori error estimates of  approximate
eigenvalues, which is essential for the design of the coefficient of
safety in practical engineering. Therefore, it is a fundamental problem to achieve lower bounds for eigenvalues of  elliptic operators.
 In fact, the study of lower bounds for  eigenvalues can date back to remarkable works of
\cite{Forsythe54,Forsythe55} and \cite{Weinberger56,Weinberger58},
where lower bounds of  eigenvalues are derived by  finite
difference methods for second order elliptic eigenvalue problems.
 Since that  finite difference methods in some sense coincide with  standard  linear
finite element methods with mass lumping, one could expect that
finite element methods with mass lumping give lower bounds for
eigenvalues of operators, we refer interested readers to \cite{AD03,HH04} for
this aspect.

Nonconforming finite element methods are alternative possible
ways to produce lower bounds for eigenvalues of  operators. In deed,
the lower bound property of eigenvalues by nonconforming elements are observed in numerics, see,
 Zienkiewicz et al.\cite{Zienkiewicz67}, for the nonconforming Morley element,
Rannacher \cite{Rannacher79}, for the nonconforming Morley and Adini
elements, Liu and Yan \cite{LiuYan05}, for the nonconforming Wilson
\cite{ShiWang10,Wilson73}, $EQ^{\rot}_1$ \cite{Ltz04}, and
$Q^{\rot}_1$ \cite{Rannacher92} elements. See,  Boffi \cite{Boffi2010}, for further remarks
on  possible properties of discrete eigenvalues produced by  nonconforming methods.

However, there are a few results to study  the lower bound property of
eigenvalues by nonconforming elements.  The first result in this
direction is analyzed in a remarkable  paper  by Armentano and Duran \cite{AD04} for the
Laplacian operator. The analysis  is based on an identity for  errors of  eigenvalues.
 It is proved that the nonconforming linear  element of  \cite{CroRav73} leads to
lower bounds for  eigenvalues provided that eigenfunctions
$u\in H^{1+r}(\Omega)\cap H_0^1(\Om)$ with $0< r<1$.  The idea is  generalized to
the enriched nonconforming rotated $Q_1$ element of \cite{Ltz04}  in  Li \cite{Li08}, and to the Wilson element in
Zhang et al. \cite{ZhangYangChen07}.  See \cite{YangZhangLin10}  for a survey of earlier works.  The extension to the
Morley element can be found in  \cite{YangLinBiLi}.  However, all of those papers
 are based  on the saturation condition of approximations by piecewise polynomials
 for which a rigorous  proof is missed in literature.  We refer interested readers to
 \cite{LinHuangLi08, LinHuangLi09, LinLin06,Yang2000,ZhangYangChen07, YangZhangLin10}
 for  expansion methods based on superconvergence or extrapolation, which analyzes the lower bound property
  of eigenvalues by nonconforming  elements  on  uniform rectangular meshes.

The aim of our paper is to introduce a new systematic method
that can produce lower bounds for eigenvalues.
The main idea is to use  nonconforming finite element methods. However,  some numerics from the literature demonstrate
that some nonconforming elements  produce upper bounds of eigenvalues though some other nonconforming elements
yield lower bounds, see\cite{LiuYan05,Rannacher79}.
We find that the general condition lies in that local approximation properties of
nonconforming finite element spaces $V_h$ should be better than  global
continuity properties of $V_h$. Then  corresponding nonconforming methods will produce
 lower bounds for eigenvalues  of elliptic operators.  More precisely,  first,  we shall  analyze errors of
 discrete eigenvalues and
eigenfunctions. Second, we shall
propose a condition on  nonconforming element methods and then under the saturation condition
prove  that it is sufficient for  lower bounds for eigenvalues.
 With this result, to obtain lower bound for
 eigenvalue is to design nonconforming  element spaces with
enough local degrees of freedom when compared to  global
continuity. This in fact results in a  systematic method for the
lower bounds of eigenvalues.  As one application of our method,  we check that this condition holds for
 most used nonconforming elements, e.g., the Wilson element
\cite{ShiWang10,Wilson73}, the nonconforming linear element by
Crouzeix and Raviart \cite{CroRav73}, the nonconforming rotated
$Q_1$ element by Rannacher and Turek \cite{Rannacher92,ShiWang10},
and the enriched nonconforming rotated $Q_1$ element by Lin, Tobiska
and Zhou \cite{Ltz04} for  second order elliptic operators, the
Morley element \cite{Morley68,ShiWang10} and the Adini element \cite{Lascaux85,ShiWang10}
for  fourth order elliptic operators, and the Morley-Wang-Xu element \cite{WX06} for  $2m$-th order elliptic
operators.   As another important application,  we follow this guidance
to enrich locally the Crouzeix-Raviart element such that  the new element satisfies  the sufficient condition and
  to propose a new  nonconforming element
method for  second order elliptic operators and show that it
actually  produces  lower bounds for eigenvalues.
As an indispensable and important part of the paper, we prove the saturation condition  for most of these nonconforming
 elements.

 The paper is organized as follows.  In the following section, we
 shall present  symmetric elliptic eigenvalue problems and  their nonconforming  element
 methods in an abstract setting. In Section 3, based on three conditions on  discrete spaces,
 we analyze  error estimates for both  discrete eigenvalues and eigenfunctions.
 In Section 4, under one more condition, we prove an abstract result
 that eigenvalues produced by nonconforming methods are
 smaller than exact ones.   In Sections 5-6, we check these
 conditions for various nonconforming methods in  literature and we also propose two new
   nonconforming methods that admit  lower bounds for
   eigenvalues in Section 7. We end this paper by Section 8 where we give some comments,
    which is followed by  appendixes where we analyze the saturation condition for piecewise polynomial approximations.

\section{Eigenvalue problems and nonconforming finite element methods}

Let $V\subset H^m(\Omega)$ denote some standard Sobolev space on
some bounded Lipschitz domain $\Omega$ in $\R^n$ with a piecewise
flat boundary $\pa\Om$.  $2m$-th order elliptic  eigenvalue
problems read: Find $(\lam, u)\in\R\x V$ such that
\begin{equation}\label{eigen}
\begin{split}
a(u, v)&=\lam (\rho u,v)_{L^2(\Om)} \text{ for any } v\in V \text{
 and  }  \|\rho^{1/2}u\|_{L^2(\Om)}=1,
\end{split}
\end{equation}
with  some  positive function $\rho\in L^{\infty}(\Om)$. The
bilinear form $a(u,v)$ is symmetric, bounded,  and coercive in the
following sense:
 \begin{equation}\label{eq1.2}
 a(u,v)=a(v,u),   |a(u,v)|\lesssim
 \|u\|_V\|v\|_V,  \text{ and } \|v\|_{V}^2\lesssim a(v,v) \text{ for any }u, v\in V,
 \end{equation}
 with the norm $\|\cdot\|_V$ over the space $V$.  Throughout the paper,
an inequality $A \lesssim B$ replaces $A\leq C\, B$ with some
multiplicative mesh-size independent constant $C>0$ that depends
only on the domain $\Omega$, the shape (e.g., through the aspect
ratio) of elements, and possibly some norm of eigenfunctions $u$.
Finally, $A \approx B$ abbreviates $A \lesssim B \lesssim A$.

Under  the conditions \eqref{eq1.2},  we have that the eigenvalue
problem \eqref{eigen}  has a sequence of eigenvalues
\begin{equation*}
0<\lam_1\leq \lam_2\leq \lam_3\leq\cdots\nearrow +\infty,
\end{equation*}
and   corresponding eigenfunctions
\begin{equation}
u_1, u_2, u_3, \cdots,
\end{equation}
which can be chosen to satisfy
\begin{equation}
(\rho u_i, u_j)_{L^2(\Om)}=\delta_{ij}, i, j=1, 2, \cdots.
\end{equation}
We define
\begin{equation}
E_{\ell}=\sspan\{u_1,u_2, \cdots, u_{\ell}\}.
\end{equation}
Then,   eigenvalues and eigenfunctions satisfy the following
well-known minimum-maximum principle:
\begin{equation}\label{minmax}
\lam_k=\min\limits_{\dim V_k=k, V_k\subset V}\max\limits_{v\in
V_k}\frac{a(v,v)}{(\rho v,v)_{L^2(\Om)}}=\max\limits_{u\in
E_{k}}\frac{a(u,u)}{(\rho u,u)_{L^2(\Om)}}.
\end{equation}
For any eigenvalue $\lam$ of \eqref{eigen}, we define
\begin{equation}
M(\lam):=\{u: u \text{ is an eigenfunction of \eqref{eigen} to }\lam
\}.
\end{equation}
We shall be interested in approximating the eigenvalue problem
\eqref{eigen} by finite element methods.  To this  end, we suppose
we are given a discrete space $V_h$ defined over  a regular
triangulation $\cT_h$ of $\overline{\Omega}$ into (closed) simplexes
or $n$-rectangles \cite{BrennerScott}.

We need the piecewise  counterparts of differential operators with respect to $\cT_h$.
 For any differential operator $\mathcal{L}$, we define its piecewise
 counterpart $\mathcal{L}_h$ in the following way: we suppose that $v_K$ is defined over $K\in\cT_h$ and
 that the differential action $\mathcal{L}v_K$ is well-defined on $K$ which is  denoted  by
 $\mathcal{L}_Kv_K$ for any $K\in\cT_h$; then we  define $v_h$ by $v_h|_K=v_K$ where $v_h|_K$ denotes its restriction of $v_h$ over $K$;
 finally we define $\mathcal{L}_hv_h$ by $(\mathcal{L}_hv_h)|_K=\mathcal{L}_Kv_K$.

 We consider the discrete
eigenvalue problem: Find $(\lam_h, u_h)\in \R\x V_h$ such that
\begin{equation}\label{discreteeigentotal}
\begin{split}
a_h(u_h, v_h)&=\lam_h (\rho u_h,v_h)_{L^2(\Om)} \text{ for any }
v_h\in V_h \text{ and }  \|\rho^{1/2}u_h\|_{L^2(\Om)}=1.
\end{split}
\end{equation}
Here and throughout of this paper, $a_h(\cdot, \cdot)$ is the
piecewise counterpart of the bilinear form $a(\cdot, \cdot)$
 where  differential operators are replaced by their discrete counterparts.
Conditions on the approximation and  continuity properties of
discrete spaces  $V_h$ are assumed  as follows, respectively.\\

(H1)  $\|\cdot\|_h:=a_h(\cdot,\cdot)^{1/2}$ is a norm over
discrete  spaces $V_h$.\\

(H2) Suppose $v\in V\cap H^{m+\mathcal{S}}(\Om)$ with $0<\mathcal{S}\leq 1$.  Then,
\begin{equation*}
\inf\limits_{v_h\in V_h}\|v-v_h\|_h\lesssim
h^{\mathcal{S}}|v|_{H^{m+\mathcal{S}}(\Om)}.
\end{equation*}

 (H3) Suppose $v\in V\cap H^{m+s}(\Om)$ with  $0<s\leq \mathcal{S}\leq 1$.  Then,
\begin{equation*}
  \sup\limits_{0\not=v_h\in V_h}\frac{a_h(v,v_h)-(Av, v_h)_{L^2(\Om)}}{\|v_h\|_h}
  \lesssim h^{s} |v|_{H^{m+s}(\Om)}\,.
\end{equation*}

(H4)~~~ Let  $u$ and $u_h$ be  eigenfunctions of problems
\eqref{eigen} and \eqref{discreteeigentotal}, respectively.
 Assume that there exists an interpolation $\Pi_hu\in V_h$ with the
following properties:
\begin{equation}\label{estimateH2}
\begin{split}
& a_h(u-\Pi_hu, u_h)=0, \\[0.5ex]
&\|\rho^{1/2}u\|_{L^2(\Om)}^2-\|\rho^{1/2}\Pi_hu\|_{L^2(\Om)}^2\lesssim
h^{2s+\triangle s},\\[0.5ex]
&\|\rho^{1/2}(\Pi_hu-u)\|_{L^2(\Om)}\lesssim
h^{\mathcal{S}+\triangle \mathcal{S}}\,,
 \end{split}
\end{equation}
when the meshsize $h$ is small enough and $u\in V\cap
H^{m+\mathcal{S}}(\Om)$ with  two positive constants $\triangle s$ and $\triangle\mathcal{S}$.

Let $N=\dim  V_h$.  Under the condition (H1), the discrete problem
\eqref{discreteeigentotal} admits a sequence of discrete eigenvalues
$$
0<\lam_{1,h}\leq\lam_{2,h}\leq\cdots\leq\lam_{N,h},
$$
and corresponding eigenfunctions
$$
u_{1,h}, u_{2,h}, \cdots, u_{N,h}\,.
$$
In the case where $V_h$ is a conforming approximation in the sense
$V_h\subset V$,  it immediately  follows from the minimum-maximum
principle \eqref{minmax} that
$$
\lam_{k}\leq \lam_{k,h}, k=1,2,\cdots, N,
$$
which indicates that  $\lam_{k,h}$ is an approximation above 
$\lam_k$.

We define the discrete counterpart of $E_{\ell}$ by
\begin{equation}
E_{\ell, h}=\sspan\{u_{1,h},u_{2,h}, \cdots, u_{\ell, h}\}.
\end{equation}
Then,  we have the following discrete  minimum-maximum principle:
\begin{equation}\label{Discreteminmax}
\lam_{k,h}=\min\limits_{\dim V_{k,h}=k, V_{k,h}\subset
V_h}\max\limits_{v\in V_{k,h}}\frac{a_h(v,v)}{(\rho
v,v)_{L^2(\Om)}}=\max\limits_{u\in E_{k,h}}\frac{a_h(u,u)}{(\rho
u,u)_{L^2(\Om)}}.
\end{equation}

\section{Error estimates of eigenvalues and eigenfunctions}

In this section, we shall analyze errors of discrete eigenvalues
and eigenfunctions by nonconforming methods. We refer to
\cite{BA91,Rannacher79} for some  alternative analysis in the
functional analysis setting.  We would like to  stress the analysis
is a nontrivial extension to  nonconforming methods of the
analysis for conforming methods in \cite{StrangeFix73}. For
simplicity of presentation, we only consider the case where
$\lambda_{\ell}$ is an  eigenvalue of multiplicity $1$ and also note that the extension
to the  multiplicity $\geq 2$  case follows by using notations and concepts, for instance, from \cite[Page 406]{CarstensenGedicke11}.

Associated with the bilinear form $a(\cdot, \cdot)$, we define the
operator $A$ by
\begin{equation}\label{DifferentialOperator}
a(u,v)=(Au, v)_{L^2(\Om)} \text{ for any }v\in V\,.
\end{equation}

Given any $f\in L^2(\Om)$, let $u_f$  be the solution to the dual
problem: Find $u_f\in V$ such that
\begin{equation}\label{Regularity}
a(u_f,v)=(\rho f,v)_{L^2(\Om)} \text{ for any }v\in V\,.
\end{equation}
Generally speaking, the regularity of $u_f$ depends on, among others,  regularities of $f$ and $\rho$,
 elliptic operators under consideration, the shape of the domain $\Omega$ and  the boundary condition imposed.
   To fix the main idea and therefore avoid too technical notation,
  throughout this paper, without loss of generality, assume that $u_f\in V\cap H^{m+s}(\Om)$ with $0<s\leq 1$ in
the sense that
\begin{equation}
\|u_f\|_{H^{m+s}(\Om)}\lesssim \|\rho^{1/2}f\|_{L^2(\Om)}\,.
\end{equation}

In order to analyze $L^2$ error estimates of eigenfunctions,  define  quasi-Ritz-projections
$P_h^{\prime}u_{\ell}\in V_h$ by
\begin{equation}\label{Firstprojection}
a_h(P_h^{\prime}u_{\ell}, v_h)=\lam_{\ell}(\rho u_{\ell},
v_h)_{L^2(\Om)}\text{ for any }v_h\in V_h.
\end{equation}

The analysis also needs   Galerkin projection operators
$P_h: V\rightarrow V_h$ by
\begin{equation}\label{Galerkin}
a_h(P_hv,w_h)=a_h(v,w_h)\text{ for any }w_h\in V_h, v\in V.
\end{equation}
\begin{remark}
We note that $P_h^{\prime}$ is identical to $P_h$ for  conforming
methods, which indicates  the difference between  conforming
elements analyzed in \cite{StrangeFix73} and nonconforming
elements under consideration.
\end{remark}
Under the conditions (H1), (H2),  and (H3),  a standard argument for
nonconforming finite element methods, see, for instance,
\cite{BrennerScott}, proves
\begin{equation}
\|\rho^{1/2}(v-P_hv)\|_{L^2(\Om)}+h^{s}\|v-P_hv\|_h\lesssim
h^{2s}|v|_{H^{m+s}(\Om)}\,,
\end{equation}
provided that $v\in V\cap H^{m+s}(\Om)$ with $0<s\leq 1$.

Throughout this paper, $u_{\ell}$, $u_{j}$, and $u_{i}$  are eigenfunctions of the problem \eqref{eigen},
while $u_{\ell, h}$, $u_{j, h}$, and $u_{i,h}$ are discrete eigenfunctions of the discrete eigenvalue problem.
Note that $P_h^{\prime}u_{\ell}$ is the finite element approximation
of $u_{\ell}$.  Under conditions (H1)-(H3), a standard argument for
nonconforming finite element methods, see, for instance,
\cite{BrennerScott}, proves
\begin{lemma}  Suppose that the conditions  (H1)-(H3) hold. Then,
\begin{equation}
\|\rho^{1/2}(u_{\ell}-P_h^{\prime}u_{\ell})\|_{L^2(\Om)}+h^{s}\|u_{\ell}-P_h^{\prime}u_{\ell}\|_h\lesssim
h^{2s}|u_{\ell}|_{H^{m+s}(\Om)}\,,
\end{equation}
provided that $u_\ell\in V\cap H^{m+s}(\Om)$ with $0<s\leq 1$.
\end{lemma}
From  $P_h^{\prime}u_{\ell}\in V_h$ we have
\begin{equation}
P_h^{\prime}u_{\ell}=\sum\limits_{j=1}^N(\rho P_h^{\prime}u_{\ell},
u_{j,h})u_{j,h}\,.
\end{equation}
For the projection operator $P_h^{\prime}$, we have the following
important property
\begin{equation}
\begin{split}
&(\lam_{j,h}-\lam_{\ell})(\rho P_h^{\prime}u_{\ell},
u_{j,h})_{L^2(\Om)}=\lam_{\ell}(\rho(u_{\ell}-P_h^{\prime}u_{\ell}),
u_{j,h})_{L^2(\Om)}\,.
\end{split}
\end{equation}
In fact, we have
\begin{equation}
\lam_{j,h}(\rho P_h^{\prime}u_{\ell},
u_{j,h})_{L^2(\Om)}=a_h(u_{j,h},P_h^{\prime}u_{\ell})=\lam_{\ell}(\rho
u_{\ell}, u_{j,h})_{L^2(\Om)}\,.
\end{equation}
Suppose that $\lam_{\ell}\not=\lam_j$ if $\ell\not=j$. Then there
exists a separation  constant $d_{\ell}$ with
\begin{equation}
\frac{\lam_{\ell}}{|\lam_{j,h}-\lam_{\ell}|}\leq d_{\ell}\text{ for
any }j\not=\ell,
\end{equation}
provided that the meshsize $h$ is small enough.

\begin{theorem}\label{L2} Let $u_{\ell}$ and $u_{\ell,h}$ be
eigenfunctions of \eqref{eigen} and \eqref{discreteeigentotal},
respectively. Suppose that  the conditions (H1)-(H3) hold.  Then,
\begin{equation}\label{eigenfunctionerrorab}
\begin{split}
\|\rho^{1/2}(u_{\ell}-u_{\ell,h})\|_{L^2(\Om)}
 \lesssim h^{2s}|u_{\ell}|_{H^{m+s}(\Om)}\,,
\end{split}
\end{equation}
provided that $u_\ell\in V\cap H^{m+s}(\Om)$ with $0<s\leq 1$.
\end{theorem}
\begin{proof}
We denote the key coefficient $(\rho P_h^{\prime}u_{\ell},
u_{\ell,h})_{L^2(\Om)}$ by $\beta_{\ell}$.  The rest can be bounded
as follows:
\begin{equation}
\begin{split}
\|\rho^{1/2}(P_h^{\prime}u_{\ell}-\beta_{\ell}u_{\ell,h})\|_{L^2(\Om)}^2&=\sum\limits_{j\not=\ell}(\rho
P_h^{\prime}u_{\ell}, u_{j,h})_{L^2(\Om)}^2\leq
d_{\ell}^2\sum\limits_{j\not=\ell}(\rho(u_{\ell}-P_h^{\prime}u_{\ell}),
u_{j,h})_{L^2(\Om)}^2\\[0.5ex]
&\leq
d_{\ell}^2\|\rho^{1/2}(u_{\ell}-P_h^{\prime}u_{\ell})\|_{L^2(\Om)}^2\,.
\end{split}
\end{equation}
This leads to
\begin{equation}
\begin{split}
&\|\rho^{1/2}(u_{\ell}-\beta_{\ell}u_{\ell,h})\|_{L^2(\Om)} \leq
\|\rho^{1/2}(u_{\ell}-P_h^{\prime}u_{\ell})\|_{L^2(\Om)}
+\|\rho^{1/2}(P_h^{\prime}u_{\ell}-\beta_{\ell}u_{\ell,h})\|_{L^2(\Om)}\\[0.5ex]
&\leq
(1+d_{\ell})\|\rho^{1/2}(u_{\ell}-P_h^{\prime}u_{\ell})\|_{L^2(\Om)}\lesssim
h^{2s}\lam_{\ell}^{(m+s)/2m}\,.
\end{split}
\end{equation}
\begin{equation}
\begin{split}
&\|\rho^{1/2}u_{\ell}\|_{L^2(\Om)}-\|\rho^{1/2}(u_{\ell}-\beta_{\ell}u_{\ell,h})\|_{L^2(\Om)}
\leq \|\rho^{1/2}\beta_{\ell}u_{\ell,h}\|_{L^2(\Om)}\\[0.5ex]
&\leq \|\rho^{1/2}u_{\ell}\|_{L^2(\Om)}+
\|\rho^{1/2}(u_{\ell}-\beta_{\ell}u_{\ell,h})\|_{L^2(\Om)}.
\end{split}
\end{equation}
Since both $u_{\ell}$ and $u_{\ell,h}$ are unit vectors, we can
choose them such that $\beta_{\ell}\geq 0$. Hence we have
$|\beta_{\ell}-1|\leq
\|\rho^{1/2}(u_{\ell}-\beta_{\ell}u_{\ell,h})\|_{L^2(\Om)}$. Thus,
we obtain
\begin{equation}\label{eigenfunctionerrorL2ab}
\begin{split}
\|\rho^{1/2}(u_{\ell}-u_{\ell,h})\|_{L^2(\Om)}
 &\leq \|\rho^{1/2}(u_{\ell}-\beta_{\ell}u_{\ell,h})\|_{L^2(\Om)}
 +|\beta_{\ell}-1|\|\rho^{1/2}u_{\ell,h}\|_{L^2(\Om)}\\[0.5ex]
 &\leq 2
 \|\rho^{1/2}(u_{\ell}-\beta_{\ell}u_{\ell,h})\|_{L^2(\Om)}\lesssim
h^{2s}|u_{\ell}|_{H^{m+s}(\Om)}\,.
 \end{split}
\end{equation}
This completes the proof.
\end{proof}

Next we analyze  errors of  eigenvalues.  To this end,  define
$\tilde{u}_{\ell,h}\in V$ by
 \begin{equation}
 a(\tilde{u}_{\ell,h}, v)=\lambda_{\ell,h}(\rho u_{\ell,h}, v)_{L^2(\Om)}\text{ for any }v\in V.
 \end{equation}
 It follows from \eqref{eigen} and \eqref{discreteeigentotal} that
\begin{equation}
\begin{split}
(\rho(\tilde{u}_{\ell,h}-u_{\ell,h}),
u_{\ell})_{L^2(\Omega)}&=\lambda_{\ell}^{-1}\lambda_{\ell,h}(\rho
u_{\ell,h},u_{\ell})_{L^2(\Omega)}-(\rho u_{\ell,h}, u_{\ell})_{L^2(\Omega)}\\[0.5ex]
&=\frac{(\lambda_{\ell,h}-\lambda_{\ell})(\rho u_{\ell,h},
u_{\ell})_{L^2(\Omega)}}{\lambda_{\ell}}.
\end{split}
\end{equation}
Thus we have
 \begin{equation}\label{Eigendifference}
 \lambda_{\ell,h}-\lambda_{\ell}=\frac{\lambda_{\ell}(\rho(\tilde{u}_{\ell,h}-u_{\ell,h}),
u_{\ell})_{L^2(\Omega)}}{(\rho u_{\ell,h}, u_{\ell})_{L^2(\Omega)}}.
 \end{equation}
Assume that $\tilde{u}_{\ell, h}\in
V\cap H^{m+s}(\Omega)$ with $0<s\leq 1$ in the sense that
\begin{equation}
\|\tilde{u}_{\ell,h}\|_{H^{m+s}(\Omega)}\lesssim
\lambda_{\ell,h}\|\rho^{1/2}u_{\ell,h}\|_{L^2(\Omega)}.
\end{equation}
Note that  $u_{\ell,h}$ is the finite element approximation of
$\tilde{u}_{\ell, h}$. A standard argument for  nonconforming
finite element methods, see, for instance,
\cite{BrennerScott}, proves
\begin{lemma} Suppose that the  conditions  (H1)-(H3) hold. Then,
\begin{equation}
\|\rho^{1/2}(u_{\ell, h}-\tilde{u}_{\ell,
h})\|_{L^2(\Om)}+h^{s}\|u_{\ell, h}-\tilde{u}_{\ell,
h}\|_h\lesssim
\lambda_{\ell,h}h^{2s}\|\rho^{1/2}u_{\ell,h}\|_{L^2(\Omega)}\,.
\end{equation}
\end{lemma}
Inserting the above estimate into \eqref{Eigendifference} proves:

\begin{theorem}\label{eigvalueerrorab} Let $\lam_{\ell}$ and $\lam_{\ell,h}$ be
eigenvalues of \eqref{eigen} and \eqref{discreteeigentotal},
respectively. Suppose that (H1)-(H3) hold.  Then,
\begin{equation}
|\lam_{\ell,h}-\lam_{\ell}|\lesssim h^{2s}|u_{\ell}|_{H^{m+s}(\Om)}\,,
\end{equation}
provided that $h$ is small enough and that $u_\ell\in V\cap H^{m+s}(\Om)$ with $0<s\leq 1$..
\end{theorem}

Finally we can  have  error estimates in the energy norm of eigenfunctions.

\begin{theorem}\label{Energynormab} Let $u_{\ell}$ and $u_{\ell,h}$ be
eigenfunctions of \eqref{eigen} and \eqref{discreteeigentotal},
respectively. Suppose that the conditions (H1)-(H3) hold.  Then,
\begin{equation}
 \|u_{\ell}-u_{\ell,h}\|_h\lesssim
h^{s}|u_{\ell}|_{H^{m+s}(\Om)}\,,
\end{equation}
provided that $u_\ell\in V\cap H^{m+s}(\Om)$ with $0<s\leq 1$.
\end{theorem}
\begin{proof}
In order to bound  errors of eigenfunctions in the energy norm,
we need the following decomposition:
\begin{equation}
\begin{split}
&a_h(u_{\ell}-u_{\ell,h}, u_{\ell}-u_{\ell,h})=a(u_{\ell},
u_{\ell})+a_h(u_{\ell,h}, u_{\ell,h})-2a_h(u_{\ell}, u_{\ell,h})
\\[0.5ex]
&=\lam_{\ell}\|\rho^{1/2}(u_{\ell}-u_{\ell,h})\|_{L^2(\Om)}^2+\lam_{\ell,h}-\lam_{\ell}
  +2\lam_{\ell}(\rho u_{\ell}, u_{\ell,h}-u_{\ell})-2a_h(u_{\ell},
  u_{\ell,h}-u_{\ell})\,.
\end{split}
\end{equation}
Then, the desired result follows from Theorem \ref{eigvalueerrorab},
\eqref{eigenfunctionerrorL2ab}, and the condition (H3).
\end{proof}

\section{Lower bounds for eigenvalues: an abstract theory}
This section proves that  the conditions (H1)-(H4) are sufficient conditions to guarantee
nonconforming finite element methods to yield lower bounds for eigenvalues of elliptic
operators.

\begin{theorem}\label{main}  Let $(\lam,u)$ and $(\lam_{h}, u_{h})$
be  solutions of problems  \eqref{eigen} and
\eqref{discreteeigentotal}, respectively.  Assume that  $u\in V\cap
\in H^{m+\mathcal{S}}(\Om)$ and that $h^{2s}\lesssim \|u-u_{h}\|_h^2$
with $0<s\leq\mathcal{S}\leq 1$. If the conditions (H1)--(H4) hold, then
\begin{equation}\label{LowerBound}
\lam_h\leq\lam,
\end{equation}
provided that $h$ is small enough.
\end{theorem}
\begin{proof} Let $\Pi_h$ be the operator in the condition (H4). A similar argument of \cite{AD04} proves
\begin{equation}\label{expansion}
\begin{split}
 \lam-\lam_{h}&=\|u-u_{h}\|_h^2-\lam_{h}\|\rho^{1/2}(\Pi_hu-u_{h})\|_{L^2(\Om)}^2\\[0.5ex]
 &\quad +\lam_{h}(\|\rho^{1/2}\Pi_h u\|_{L^2(\Om)}^2-\|\rho^{1/2}u\|_{L^2(\Om)}^2).
\end{split}
\end{equation}
(We refer interested readers to Zhang et al. \cite{ZhangYangChen07} for an identity with full terms). From the abstract error estimate \eqref{eigenfunctionerrorab} it follows that
\begin{equation}
\|\rho^{1/2}(u-u_h)\|_{L^2(\Om)}\lesssim h^{2s}\,.
\end{equation}
Hence the triangle inequality and (H4) plus the saturation condition $h^{2s}\lesssim \|u-u_{h}\|_h^2$ show that the second
third term on the right-hand side of \eqref{expansion} is of
higher order than the first term.  If $\|\rho^{1/2}\Pi_h u\|_{L^2(\Om)}^2\leq \|\rho^{1/2}u\|_{L^2(\Om)}^2$, then the condition states
that the third term is of higher order than the first term; otherwise, it will be positive. This completes the proof.
\end{proof}
The condition that $h^{2s}\lesssim \|u-u_{h}\|_h^2$
is usually  referred to  as the  saturation condition in the literature.
The condition is closely related to the inverse theorem in the context of the approximation theory by  trigonometric
polynomials or splines. For the approximation by  conforming piecewise polynomials, the inverse theorem was analyzed  in \cite{BabuskaKellog1979,Widlund1977}. For  nonconforming finite element methods,  the saturation condition was first analyzed in Shi \cite{Shi1986}
  for the Wilson  element by an example, which was developed by Chen and Li \cite{ChenLi1994} by an expansion of the error.
  See \cite{KrizekRoosChen2011}  for lower bounds of  discretization errors by conforming linear/bilinear finite elements.   Babuska and Strouboulis \cite{BabuskaStrouboulis2000} analyzed Lagrange finite element methods  for  elliptic problems in  one dimension.  In  appendixes,
 we shall analyze the saturation condition for most of nonconforming finite element methods under consideration.  To our knowledge,  it is the first time   to analyze  systematically  this condition for nonconforming methods.

Since  Galerkin projection operators $P_h$  from
\eqref{Galerkin} or  their high order perturbations  of  nonconforming spaces $V_h$ are taken  as
interpolation operators $\Pi_h$,  their error estimates are dependent on only local approximation properties but not global continuity
 properties of   spaces $V_h$ while  errors $\|u-u_h\|_h$  generally depend on both properties (see Theorems \ref{L2} and \ref{Energynormab} ).
 On the other hand, the term $\|\rho^{1/2}u\|_{L^2(\Om)}^2-\|\rho^{1/2}\Pi_hu\|_{L^2(\Om)}^2$ will be either of high order or negative when we have enough many local degrees of freedom (compared to the continuity) and therefore  consistency errors in $\|u-u_h\|_h^2$ will be dominant in the sense that
 $$
\|u-u_h\|_h^2\geq \|\rho^{1/2}u\|_{L^2(\Om)}^2-\|\rho^{1/2}\Pi_hu\|_{L^2(\Om)}^2.
$$
If this happens we say  local approximation properties of  spaces $V_h$ are better than
global  continuity properties of $V_h$.  Hence, Theorem \ref{main} states that   corresponding methods of
 eigenvalue problems will produce  lower bounds for
eigenvalues for this situation. Thus, to get a lower bound  for an eigenvalue is to design
nonconforming finite element spaces with enough local degrees of
freedom when compared to global continuity properties of $V_h$.
This in fact  provides a systematic tool for the construction of
lower bounds  for eigenvalues of operators in mathematical
science.

\section{Nonconforming elements of second order elliptic operators}
This section presents  some nonconforming schemes of second order elliptic eigenvalue problems that the conditions
(H1)-(H4) proposed in Section 2 are satisfied. Let the
boundary $\pa\Om$ be divided into two parts: $\Gamma_D$ and
$\Gamma_N$ with $|\Gamma_D|>0$, and $\Gamma_D\cup\Gamma_N=\pa\Om$.
For ease of presentation, assume that \eqref{eigen} is the Poisson eigenvalue problem imposed  general boundary
conditions.

Let $\mathcal{T}_h$  be  regular $n$-rectangular triangulations of
  domains $\Om\subset\R^n$ with $2\leq n$ in the sense that
$\bigcup_{K\in \mathcal{T}_h}K=\bar{\Omega}$, two distinct elements
$K$ and $K'$ in $\mathcal{T}_h$ are either disjoint, or share an
$\ell$-dimensional hyper-plane, $\ell=0, \cdots, n-1$. Let
$\mathcal{H}_h$ denote the set of all $n-1$ dimensional hyper-planes
in $\mathcal{T}_h$ with the set of interior
  $n-1$ dimensional hyper-planes $\mathcal{H}_h(\Omega)$ and the set of
boundary  $n-1$ dimensional hyper-planes $\mathcal{H}_h(\pa\Omega)$.  $\mathcal{N}_h$ is the set of nodes of $\mathcal{T}_h$
with the set of  internal nodes $\mathcal{N}_h(\Omega)$ and the set
of boundary nodes $\mathcal{N}_h(\pa \Omega)$.

For each $K\in \mathcal{T}_h$,  introduce the following affine
invertible transformation
$$
F_K:\hK\rightarrow K, x_i=h_{x_i,K}\xi_i+x_i^{0}
$$
with  the center  $(x_1^{0}, x_2^0, \cdots, x_n^0)$ and the
  lengths  $2h_{x_i,K}$ of $K$ in the directions of the $x_i$-axis,
  and the reference element
$\hK=[-1,1]^n$.  In addition, set $h=\max_{1\leq i\leq n}
h_{x_i}$.

Over the above mesh $\cT_h$, we shall consider two classes of
nonconforming element methods for the eigenvalue problem
\eqref{eigen}, namely, the Wilson element in any dimension,
the enriched nonconforming  rotated $Q_1$ element in any dimension.

Let $V_h$  be  discrete spaces of aforementioned
nonconforming element methods.  The finite  element approximation of
Problem \eqref{eigen} is defined as in \eqref{discreteeigentotal}.

For all the elements, one can use  continuity and boundary
conditions for  discrete  spaces $V_h$ given below  to verify the conditions
 (H1)-(H3), see \cite{Ltz04, Rannacher92, ShiWang10,Wilson73} for further details. Let $(\lam, u)$ and
$(\lam_h, u_h)$ be solutions to  problems
\eqref{eigen} and \eqref{discreteeigentotal},  by Theorems
\ref{L2}, \ref{eigvalueerrorab}, and \ref{Energynormab} we have
\begin{equation}\label{PoissonFiniteEstimate}
|\lam-\lam_h|+h^{\mathfrak{s}}\|u-u_h\|_h+\|\rho^{1/2}(u-u_h)\|_{L^2(\Om)}\lesssim
h^{2\mathfrak{s}}\,,
\end{equation}
provided that $u\in V\cap H^{1+\mathfrak{s}}(\Om)$ with $0<\mathfrak{s}\leq 1$.

We shall analyze the key  condition (H4) for these elements in the
subsequent subsections.

\subsection{The Wilson element in any dimension}\label{sub5.1}

Denote by $Q_{Wil}(\hat{K})$ the nonconforming  Wilson  element
space \cite{ShiWang10,Wilson73} on the reference element
defined by
 \be\label{nonconformingnDWil} Q_{Wil}(\hat{K})
 =Q_1(\hK)+\sspan\{ \xi_1^{2}-1, \xi_2^{2}-1, \cdots, \xi_n^2-1\},
  \ee
 where $Q_1(\hK)$ is the  space of polynomials of degree$\leq 1$ in
 each variable.  The  nonconforming  Wilson  element space $V_h$ is then defined
as
\begin{equation}\label{WilsonnD}
 V_h:=\begin{array}[t]{l}\big\{v\in L^2(\Om):
 v|_{K}\circ F_K\in Q_{Wil}(\hK) \text{ for each }K\in \mathcal{T}_h,
 v \text{ is continuous }\\[0.5ex]
 \text{ at  internal nodes,  and vanishes at boundary nodes on
 }\Gamma_D
 \big\}\,.
 \end{array}\nn
\end{equation}
The degrees of freedom  read
 \begin{equation*}
    v(a_j),1\leq j\leq2^n \text{ and }
    \frac{1}{|K|}\int_K\frac{\partial^2v}{\partial x^2_i} d x,1\leq i\leq n,
   \end{equation*}
   where $a_j$ denote vertexes of element $K$.

In order to show the condition  (H4),  let $P_h$ be the Galerkin
projection operator defined in \eqref{Galerkin}.  The approximation property of the operator $P_h$ reads
\begin{equation}
h\|u-P_hu\|_h+\|u-P_hu\|_{L^2(\Om)} \lesssim
h^{2}|u|_{H^2(\Om)}\,,
\end{equation}
provided that $u\in V\cap H^{2}(\Om)$. This plus \eqref{PoissonFiniteEstimate} lead to
\begin{equation}\label{eq5.4}
\begin{split}
\lambda_h(\rho (P_hu-u), P_hu+u)_{L^2(\Om)}&=\lambda(\rho (P_hu-u), P_hu+u)_{L^2(\Om)}+\mathcal{O}(h^{4})\\
&=2\lambda(\rho (P_hu-u), u)_{L^2(\Om)}+\mathcal{O}(h^{4}).
\end{split}
\end{equation}
To analyze the term $\lambda(\rho (P_hu-u), u)_{L^2(\Om)}$,
let $I_h$ be the canonical interpolation operator for the Wilson element, which admits the following error estimates:
\begin{equation}\label{InterpolationErrorWilson-smooth}
h\|u-I_hu\|_h+\|u-I_hu\|_{L^2(\Om)}\lesssim h^{2+\mathfrak{s}}|u|_{H^{2+\mathfrak{s}}(\Om)},
\end{equation}
provided that $u\in V\cap H^{2+\mathfrak{s}}(\Om)$ with
$0<\mathfrak{s}\leq 1$. Since $\|u-P_hu\|_h\lesssim h^{1+\mathfrak{s}}$ provided that $u\in V\cap H^{2+\mathfrak{s}}(\Om)$ with
$0<\mathfrak{s}\leq 1$,
\begin{equation*}
\lambda(\rho u,  P_hu-I_hu)_{L^2(\Om)}-a_h(u, P_hu-I_hu)
\lesssim h|u|_{H^2(\Om)}\|P_hu-I_hu\|_h\lesssim h^{2+\mathfrak{s}}\|u\|_{H^{2+\mathfrak{s}}(\Om)}^2.
\end{equation*}
This and \eqref{InterpolationErrorWilson-smooth} state
\begin{equation}\label{eq5.6}
\begin{split}
\lambda(\rho (P_hu-u), u)_{L^2(\Om)}&=\lambda(\rho u,  P_hu-u)_{L^2(\Om)}-a_h(u, P_hu-u)+a_h(u-P_hu, P_hu-u)\\
&=a_h(u-I_hu,u)+\mathcal{O}(h^{2+\mathfrak{s}}).
\end{split}
\end{equation}
To analyze the term $a_h(u-I_hu,u)$, let $I_K$ denote the restriction of $I_h$ on element $K$.  Then we have the following
 result.
 \begin{lemma} For  any $u\in P_3(K)$ and $v\in P_1(K)$, it holds that
 \begin{equation}\label{eq5.7}
 (\nabla (u-I_Ku), \nabla v)_{L^2(K)}=-\sum\limits_{i=1}^n\sum\limits_{j\not = i}\frac{h_{x_i,K}^2}{3}\int_K \frac{\pa^3 u}{\pa x_i^2\pa x_j}\frac{\pa v}{\pa x_j}dx.
 \end{equation}
 \end{lemma}
\begin{proof} The definition of the interpolation operator $I_K$ leads to
\begin{equation*}
u-I_Ku=\sum\limits_{i=1}^n\frac{h_{x_i,K}^3}{6}\frac{\pa^3 u}{\pa x_i^3}(\xi_i^3-\xi_i)+\sum\limits_{i=1}^n\sum\limits_{j\not= i}\frac{h_{x_i,K}^2 h_{x_j,K}}{2}\frac{\pa^3 u}{\pa x_i^2\pa x_j}(\xi_i^2\xi_j-\xi_j).
\end{equation*}
A direct calculation proves
\begin{equation*}
 (\nabla (u-I_Ku), \nabla v)_{L^2(K)}=-\sum\limits_{i=1}^n\sum\limits_{j\not = i}\frac{h_{x_i,K}^2}{3}\int_K \frac{\pa^3 u}{\pa x_i^2\pa x_j}\frac{\pa v}{\pa x_j}dx,
 \end{equation*}
which completes the proof.
\end{proof}

Given any  element $K$,  define $J_{\ell,K}v\in P_\ell(K)$ by
\begin{equation}\label{QuasiInterpolation}
\int_{K}\na^{i}J_{\ell,K}v dx=\int_K \na^{i}v dx, i=0, \cdots, \ell,
\end{equation}
for any $v\in H^\ell(K)$.  Note that the operator $J_{\ell,K}$ is
well-defined.  Let $\Pi_K^0$ denote the constant projection operator over $K$, namely,
\begin{equation*}
\Pi_K^0v:=\frac{1}{|K|}\int_Kvdx \text{ for any }v\in L^2(K).
\end{equation*}
 The property of operator $J_{\ell,K}$ reads
\begin{equation}\label{eq5.9}
\|\nabla^i(v-J_{\ell,K}v)\|_{L^2(K)}\lesssim h_K^{\ell-i} \|\nabla^\ell (v-J_{\ell,K}v)\|_{L^2(K)}\text{ and } \nabla^\ell J_{\ell,K}=\Pi^0_K\nabla^\ell v \text{ for any }v\in H^\ell(K).
\end{equation}
\begin{lemma} For  uniform meshes, it holds that
\begin{equation}\label{eq5.6bb}
 (\nabla_h (u-I_hu), \nabla u)_{L^2(\Om)}=\sum\limits_{i=1}^n\sum\limits_{j\not = i}\frac{h_{x_i,K}^2}{3}\int_K \bigg(\frac{\pa^2 u}{\pa x_i\pa x_j}\bigg)^2dx+o(h^2),
 \end{equation}
 provided that $u\in H^3(\Om)$ and the meshsize is small enough.
\end{lemma}
\begin{proof}
A combination of  \eqref{InterpolationErrorWilson-smooth} and \eqref{eq5.9} leads to
\begin{equation*}
\begin{split}
&(\nabla_h(u-I_hu),\nabla u)_{L^2(\Om)}
=\sum\limits_{K\in\cT_h}(\nabla (u-I_Ku),\nabla u)_{L^2(K)}\\
&= \sum\limits_{K\in\cT_h}(\nabla (u-I_Ku),\nabla J_{1,K} u)_{L^2(K)}+\mathcal{O}(h^3).
\end{split}
\end{equation*}
The operator $J_{3,K}$ yields the following decomposition
\begin{equation}\label{eq5.10}
\begin{split}
\sum\limits_{K\in\cT_h}(\nabla (u-I_Ku),\nabla J_{1,K} u)_{L^2(K)}
& =\sum\limits_{K\in\cT_h}(\nabla (I-I_K)J_{3,K}u,\nabla J_{1,K} u)_{L^2(K)}\\
&\quad +\sum\limits_{K\in\cT_h}(\nabla (I-I_K)(I-J_{3,K})u,\nabla J_{1,K} u)_{L^2(K)}.
\end{split}
\end{equation}
It follows from \eqref{InterpolationErrorWilson-smooth} and \eqref{eq5.9} that the second term on the right-hand side of the above equation can be
estimated as
\begin{equation*}
\sum\limits_{K\in\cT_h}(\nabla (I-I_K)(I-J_{3,K})u,\nabla J_{1,K} u)_{L^2(K)}\lesssim \sum\limits_{K\in\cT_h}h_K^2\|(I-\Pi_K^0)\nabla^3u\|_{L^2(K)}\|\nabla u\|_{L^2(K)}=o(h^2),
\end{equation*}
since piecewise constant functions are dense in the space $L^2(\Omega)$ when the meshsize is small enough. The first term on the right-hand side of \eqref{eq5.10} can be analyzed by \eqref{eq5.7},
 which reads
 \begin{equation*}
 \begin{split}
 \sum\limits_{K\in\cT_h}(\nabla (I-I_K)J_{3,K}u,\nabla J_{1,K} u)_{L^2(K)}&=
-\sum\limits_{K\in\cT_h}
 \sum\limits_{i=1}^n\sum\limits_{j\not = i}\frac{h_{x_i,K}^2}{3}\int_K \frac{\pa^3 J_{3,K}u}{\pa x_i^2\pa x_j}\frac{\pa J_{1,K} u}{\pa x_j}dx\\
 &=
-\sum\limits_{K\in\cT_h}
 \sum\limits_{i=1}^n\sum\limits_{j\not = i}\frac{h_{x_i,K}^2}{3}\int_K \frac{\pa^3 u}{\pa x_i^2\pa x_j}\frac{\pa  u}{\pa x_j}dx+o(h^2),
\end{split}
 \end{equation*}
 when the meshsize is small enough.  Since the mesh is uniform and $\frac{\pa^2 u}{\pa x_i \pa x_j}\frac{\pa  u}{\pa x_j}$ vanish on the boundary
 which is perpendicular to $x_i$ axises,  elementwise integrations by parts complete the proof.
\end{proof}
A summary of \eqref{eq5.4}, \eqref{eq5.6} and \eqref{eq5.6bb} proves that
\begin{equation}
\lambda_h(\rho (P_hu-u), P_hu+u)_{L^2(\Om)}\geq 0
\end{equation}
when the meshsize is small enough and $u\in H^3(\Om)$.  In  appendix  A,  we prove that $h \lesssim \|u-u_h\|_h$
when $u\in H^{2+\mathfrak{s}}(\Om)$. Therefore, the condition (H4) holds for the Wilson element when $u\in H^3(\Om)$ and the
 mesh is uniform.

\subsection{The enriched nonconforming rotated $Q_1$ element in any
dimension}  Denote by $Q_{EQ}(K)$ the enriched nonconforming rotated $Q_1$
element space defined by \cite{Ltz04}
 \begin{equation}\label{nonconformingEQ}
Q_{EQ}({K}):=P_1(K)+\sspan\{ x_1^{2}, x_2^2, \cdots, x_n^2\}.
 \end{equation}
 The enriched nonconforming rotated
$Q_1$ element space $V_h$ is then defined by
\begin{equation}\label{EQrotated}
 V_h:=\begin{array}[t]{l}\big\{v\in L^2(\Om):
 v|_{K}\in Q_{EQ}(K) \text{ for each }K\in \cT_h,  \int_f[v]df=0,\\
 \text{ for all internal $n-1$ dimensional  hyper-planes } f\,,
 \text{ and }\int_f vdf=0 \text{ for all $f$ on  } \Gamma_D
 \big\}\,.
 \end{array}\nn
\end{equation}
Here and throughout this paper, $[v]$ denotes the jump of $v$
 across $f$. For the enriched  nonconforming rotated $Q_{1}$ element, we define
the interpolation operator  $\Pi_{h}:H_D^{1}(\Om )\rightarrow V_h$
by
\begin{equation}\label{InterpolationEQ}
\begin{split}
&\int_{f} \Pi_{h} vdf =\int_{f}v df \text{ for any } v \in
H_D^{1}(\Om ),
\text{  $f\in\mathcal{H}_h$ },\\
&\int_K \Pi_{h} v dx =\int_K  v dx  \text{ for any $K\in \cT_h$}.
\end{split}
\end{equation}
For this interpolation operator, we have
\begin{lemma}\label{InterpolationErrorEQ}  There holds that
\begin{equation}
\|u-\Pi_h u\|_{L^2(K)}\lesssim h^{2}|u|_{H^2(K)} \text{for any }
u\in H^2(K) \text{ and } K\in \cT_h,
\end{equation}
\begin{equation}
\|u-\Pi_h u\|_{L^2(K)}\lesssim h^{1+\mathfrak{s}}|u|_{H^{1+\mathfrak{s}}(K)}\text{for any
} u\in H^{1+\mathfrak{s}}(K) \text{ with } 0<\mathfrak{s}<1 \text{ and } K\in \cT_h.
\end{equation}
\end{lemma}
\begin{proof}Since $u-\Pi_h u$ has vanishing mean
on $n-1$ dimensional  hyper-plane of  $K$, it follows from the Poincare inequality that
\begin{equation*}
\|u-\Pi_hu\|_{L^2(K)}\lesssim h_K\|\na (u-\Pi_h u)\|_{L^2(K)}.
\end{equation*}
Then the desired result follows from the usual interpolation theory
and the interpolation space theory for the singular case $u\in
H^{1+\mathfrak{s}}(K)$.
\end{proof}

\begin{lemma}\label{EnrichedNQ1} For the enriched nonconforming rotated
$Q_1$ element, it holds the condition (H4).
\end{lemma}
\begin{proof} We define the space $
Q_K=\left (
\begin{array}{c}
a_{11}+a_{12} x_1\\
a_{21}+a_{22} x_2\\
\cdots \\
a_{n1}+a_{n2}x_n
\end{array}
\right ) $  with free parameters $a_{11}$, $a_{12}, \cdots, a_{n1},
a_{n2}$. From the definition of the operator $\Pi_h$,  we have
\begin{equation}
(\na (u-\Pi_hu),\bpsi)_{L^2(K)}=0,  \text{ for any } \bpsi\in Q_K.
\end{equation}
 Let $\nabla_h$ be the piecewise gradient operator which is defined element by element. Since $\na_h
\Pi_h u|_K\in Q_K$,  this leads to
\begin{equation}
(\na_h\Pi_h u)|_K=P_K(\na u|_K),
\end{equation}
with  the $L^2$ projection operator $P_K$ from $L^2(K)$ onto $Q_K$.
  This proves $a_h(u-\Pi_hu, u_h)=0$. It remains to show
  estimates in (H4).  To this end, let $\Pi^0$ be the piecewise
   constant projection operator (defined by $\Pi^0|_K=\Pi_K^0$ for element $K$ ). Without loss of generality, we assume that $\rho$
   is piecewise constant.   It follows from the definition
    of the interpolation operator $\Pi_h$ that
\begin{equation}
\begin{split}
&\|\rho^{1/2}\Pi_hu\|_{L^2(\Om)}^2-\|\rho^{1/2}u\|_{L^2(\Om)}^2=(\rho(\Pi_hu-u), \Pi_hu+u)_{L^2(\Om)}\\
&=(\rho(\Pi_hu-u), \Pi_hu+u-\Pi^0(\Pi_hu+u))_{L^2(\Om)}\\
&\lesssim h\|\rho^{1/2}(\Pi_hu-u)\|_{L^2(\Om)}\|\na_h(\Pi_hu+u)\|_{L^2(\Om)},
\end{split}
\end{equation}
which completes the proof of (H4) with $s=\triangle s=\mathcal{S}=\triangle\mathcal{S}=1$ for the case $u\in H_D^1(\Om)\cap
H^2(\Om)$; with  $s=\mathcal{S}=\mathfrak{s}$, $\triangle s=2-\mathfrak{s}$, and  $\triangle\mathcal{S}=1$,
 for the case $u\in H_D^1(\Om)\cap H^{1+\mathfrak{s}}(\Om)$ with $0<\mathfrak{s}<1$.
\end{proof}

In appendixes  A and B, we  show that $h \lesssim
\|u-u_h\|_h$ when $u\in H^{2}(\Om)$ and that there exist meshes such
that $h^{\mathfrak{s}}\lesssim \|u-u_h\|_h$ holds when $u\in H^{1+\mathfrak{s}}(\Om)$ with
$0<\mathfrak{s}< 1$.  Therefore,  we have that the result in Theorem \ref{main}
holds for this class of elements.

\section{Morley-Wang-Xu elements for $2m$-th order operators}
This section studies  $2m$-th order elliptic eigenvalue problems
defined over the bounded domain $\Om\subset\R^n$ with $1< n$ and
$m\leq n$. Let $\kappa=(\kappa_1,\cdots, \kappa_n)$ be the
multi-index with $|\kappa|=\sum\limits_{i=1}^n\kappa_i$, we define
the space
\begin{equation}
V:=\{v\in L^2(\Om), \frac{\pa^{\kappa}v}{\pa x^{\kappa}}\in
L^2(\Om),|\kappa |\leq m, v|_{\pa\Om}=\frac{\pa^{\ell}
v}{\pa\nu^{\ell}}|_{\pa\Om}=0, \ell=1,\cdots, m-1\},
\end{equation}
with $\nu$ the unit normal vector to $\pa \Om$. The partial derivatives
$\frac{\pa^{\kappa}v}{\pa x^{\kappa}}$ are  defined as
\begin{equation}
\frac{\pa^{\kappa}v}{\pa x^{\kappa}}:= \frac{\pa^{|\kappa|}v}{\pa
x_1^{\kappa_1}\cdots\pa x_n^{\kappa_n}}\,.
\end{equation}

Let $D^{\ell}v$ denote the $m$-th order tensor of all $\ell$-th
order derivatives of $v$, for instance, $\ell=1$ the gradient,  and
$\ell=2$ the Hessian matrix.   Let $\cC$ be a positive definite
operator with the same symmetry as $D^mv$,  the bilinear form
$a(u,v)$ reads
\begin{equation}
a(u,v):=(\sigma, D^{m}v)_{L^2(\Om)} \text{ and }
\sigma:=\mathcal{C}D^mu,
\end{equation}
which gives rise to  the energy norm
\begin{equation}
\|u\|_{V}^2:=a(u, u) \text{ for any }
 u\in V\,,
\end{equation}
which is equivalent to the usual  $|u|_{H^m(\Om)}$ norm  for any
$u\in V$.

$2m$-th order elliptic eigenvalue problems read: Find $(\lam,
u)\in\R\x V$ with
\begin{equation}\label{eigen2m}
\begin{split}
a(u, v)&=\lam (\rho u,v)_{L^2(\Om)} \text{ for any } v\in V \text{
and }  \|\rho^{1/2}u\|_{L^2(\Om)}=1,
\end{split}
\end{equation}
with  some  positive function $\rho\in L^{\infty}(\Om)$.

Consider  Morley-Wang-Xu elements in \cite{WX06} and apply them
to eigenvalue problems under consideration.  Let $\cT_h$ be some
shape regular decomposition into $n$-simplex of the domain $\Om$.
Denote by $\mathcal{H}_{n-i,h}$, $i=1,\cdots, n$, all $n$-$i$
dimensional subsimplexes of $\cT_h$ with $\nu_{n-i,f}$ any one of  unit
normal vectors to $f\in \mathcal{H}_{n-i,h}$. Let $[\cdot]$ denote
the jump of piecewise functions over $f$.  For any $n$-$i$ dimensional  boundary
sub-simplex $f$, the jump $[\cdot]$  denotes the trace restricted to
$f$.
 As usual, $h_K$ is the diameter of $K\in\cT_h$, and $h_f$
the diameter of $f\in\mathcal{H}_{n-i,h}$. Given $K\in\cT_h$, let
$\pa K$ denote the  boundary of $K$.
   Morley-Wang-Xu element spaces  are defined in \cite{WX06},
  which read
\begin{equation}\label{Morleyelement}
V_{h}:=\{v\in L^2(\Om), v|_{K}\in P_m(K), \int_{f}\big[\frac{\pa^{m-i}
v}{\pa \nu_{n-i,f}^{m-i}}\big]df=0, \forall f\in \mathcal{H}_{n-i,h},
 i=1, \cdots, m\}.
\end{equation}

Define the discrete stress $\sigma_h=\cC D_h^m u_h$, the broken
versions  $a_h(\cdot,\cdot)$ and $\|\cdot\|_{\cC_h}$ follow, respectively,
\begin{equation*}
\begin{split}
a_h(u_h, v_h):&=(\sigma_h,D_h^mv_h)_{L^2(\Om)},
\text{ for any }  u_h\,, v_h\in  V+V_{h}\,,\\[0.5ex]
\|u_h\|_{h}^2:&=a_h(u_h, u_h) \text{ for any }
 u_h\in V+V_{h}\,,
\end{split}
\end{equation*}
where $D_h^m$ is  defined elementwise with respect to the
partition $\cT_h$.

  The discrete eigenvalue  problem  reads:
Find $(\lambda_h, u_{h})\in \R\x V_{h}$, such that
\begin{equation}\label{Discreteeigen2m}
\begin{split}
a_h(u_h,  v_h)=\lam_h(\rho u_h,v_h)_{L^2(\Om)}\text{ for any }
v_h\in V_{h}
 \text{ and } \|\rho^{1/2}u_h\|_{L^2(\Om)}=1.
\end{split}
\end{equation}
The canonical  interpolation operator for the spaces $V_{h}$ is  defined
by: Given any $v\in V$, the interpolation $\Pi_h v\in V_{h}$ is
defined as
\begin{equation}\label{Interpolation2m-th}
\int_f \frac{\pa^{m-i} \Pi_hv}{\pa \nu_{n-i,f}^{m-i}}
df=\int_f\frac{\pa^{m-i} v}{\pa \nu_{n-i,f}^{m-i}} df, \text{ for
any } f\in\mathcal{H}_{n-i,h}\,, i=1\,,\cdots\,,m\,.
\end{equation}
For this interpolation, we have the following approximation
\begin{equation}\label{InterpolationError2m}
\|\rho^{1/2}(u-\Pi_hu)\|_{L^2(\Om)}\lesssim
h^{m+\mathfrak{s}}|u|_{H^{m+\mathfrak{s}}(\Om)}\text{ for any }u\in V\cap H^{m+\mathfrak{s}}(\Om)
\text{ with } 0<\mathfrak{s}\leq 1\,.
\end{equation}
 It is straightforward to see that conditions  (H1)-(H3) hold for this class of elements, see \cite{ShiWang10,WX06}.
Then, it follows from Theorems \ref{L2}, \ref{eigvalueerrorab},
and \ref{Energynormab} that
\begin{equation}\label{FiniteElementError2m}
\|u-u_h\|_h\lesssim h^{\mathfrak{s}} \text{ and } \|u-u_h\|_{L^2(\Om)}\lesssim
h^{2\mathfrak{s}}\,,
\end{equation}
provided that eigenfunctions $u\in V\cap H^{m+\mathfrak{s}}(\Om)$ with
$0<\mathfrak{s}\leq 1$.

\begin{theorem}\label{Theorem7.1}  Let $(\lam,u)$ and $(\lam_{h}, u_{h})$
be  solutions of problems  \eqref{eigen2m} and
\eqref{Discreteeigen2m}, respectively.  Then,
\begin{equation}
\lam_h\leq\lam,
\end{equation}
provided that $h$ is small enough.
\end{theorem}
\begin{proof} The definition of  $\Pi_h$ in
\eqref{Interpolation2m-th} yields  $a_h(u-\Pi_hu, v_h)=0$ for any
$v_h\in V_h$.  The condition (H4) follows immediately from
  \eqref{InterpolationError2m}.  In addition, in appendixes A and B, we show that $h\lesssim
|u-u_{h}|_h$  when $u \in
 V\cap H^{m+1}(\Om)$ and that  there exist
meshes such that $h^{\mathfrak{s}}\lesssim \|u-u_h\|_h$ holds when $u\in
H^{m+\mathfrak{s}}(\Om)$ with $0<\mathfrak{s}< 1$. Then, the desired result follows from
Theorem \ref{main} for  $m\geq 2$.
\end{proof}

\section{New nonconforming elements}
In this section, we shall follow the condition (H4) and the
saturation condition in Theorem \ref{main} to propose two  new
nonconforming finite elements for
 second order operators.   This is of two fold, one is to modify
 a nonconforming element in  literature such that the
 modified  one will meet the condition (H4),  the other is to
 construct a new nonconforming element.

\subsection{The enriched Crouzeix-Raviart element}  To fix the idea, we only  consider
 the case where $n=2$ and note that the results can be generalized to any dimension.
 Let $\cT_h$ be some shape regular decomposition into
 triangles of the polygonal domain $\Om\subset\R^2$. Here we  restrict ourselves to the case where the
bilinear form $a(u,v)=(\na u, \na v)_{L^2(\Om)}$ with the mixed
boundary condition  $|\Gamma_N|\not=0$.

Note that the original Crouzeix-Raviart element can only guarantee
theoretically   lower bounds of eigenvalues for the singular
case in the sense that $u\in H^{1+\mathfrak{s}}(\Omega)$ with $0<\mathfrak{s}<1$.  To
produce  lower bounds of  eigenvalues for both the singular case
$u\in H^{1+\mathfrak{s}}(\Omega)$  and the smooth case $u\in H^2(\Omega)$, we
propose to enrich the shape function space by $x_1^2+x_2^2$ on each
element. This leads to the following shape function space
 \begin{equation}\label{nonconformingECR}
Q_{ECR}(K)=P_1(K)+\sspan\{x_1^2+x_2^2\}\quad \text{ for any
}K\in\cT_h.
 \end{equation}
  The enriched
 Crouzeix-Raviart element space $V_h$ is then defined by
\begin{equation}\label{ECR}
 V_h:=\begin{array}[t]{l}\big\{v\in L^2(\Om):
 v|_{K}\in Q_{ECR}(K) \text{ for each }K\in \cT_h,
 \int_f[v]df=0,\\
 \text{ for all internal edges~} f\,,
 \text{ and }\int_f vdf=0 \text{ for all edges  $f$ on  } \Gamma_D
 \big\}\,.
 \end{array}\nn
\end{equation}
For the enriched  Crouzeix-Raviart  element, we define the
interpolation operator  $\Pi_{h}:H_D^{1}(\Om )\rightarrow V_h$ by
\begin{equation}\label{InterpolationECR}
\begin{split}
&\int_{f} \Pi_{h} vdf =\int_{f}v df \text{ for any } v \in
H_D^{1}(\Om )
\text{ for any edge $f$ },\\
&\int_K \Pi_{h} v dx =\int_K  v dx  \text{ for any $K\in \cT_h$}.
\end{split}
\end{equation}
For this interpolation operator, a similar argument of Lemma
\ref{InterpolationErrorEQ} leads to:
\begin{lemma}\label{InterpolationErrorECR}  There holds that
\begin{equation}
\|u-\Pi_h u\|_{L^2(K)}\lesssim h^{2}|u|_{H^2(K)} \text{for any }
u\in H^2(K) \text{ and } K\in \cT_h,
\end{equation}
\begin{equation}
\|u-\Pi_h u\|_{L^2(K)}\lesssim h^{1+s}|u|_{H^{1+\mathfrak{s}}(K)}\text{for any
} u\in H^{1+\mathfrak{s}}(K) \text{ with } 0<\mathfrak{s}<1 \text{ and } K\in \cT_h.
\end{equation}
\end{lemma}

\begin{lemma} For the enriched Crouzeix-Raviart element, it holds the condition (H4).
\end{lemma}
\begin{proof} We follow the idea in Lemma \ref{EnrichedNQ1} to define the space $
Q_K=\left (
\begin{array}{c}
a_{11}+a_{12} x_1\\
a_{21}+a_{12} x_2\\
\end{array}
\right ) $ with free parameters $a_{11}$, $a_{21},  a_{12}$. From
the definition of the operator $\Pi_h$,  we have
\begin{equation}\label{localprojection}
(\na (u-\Pi_hu),\bpsi)_{L^2(K)}=0,  \text{ for any } \bpsi\in Q_K.
\end{equation}
Indeed, we integrate by parts to get
\begin{equation*}
(\na (u-\Pi_hu),\bpsi)_{L^2(K)}=-(u-\Pi_hu,
\div\bpsi)_{L^2(K)}+\sum\limits_{f\subset\partial
K}\int_{f}(u-\Pi_hu)\bpsi\cdot\nu_fds.
\end{equation*}
Since $\div\bpsi$ and $\bpsi\cdot\nu_f$ (on each edge $f$) are
constant, then \eqref{localprojection} follows from
\eqref{InterpolationECR}. Since $\na_h \Pi_h u|_K\in Q_K$, the
identity \eqref{localprojection} leads to
\begin{equation}
(\na_h\Pi_h u)|_K=P_K(\na u|_K),
\end{equation}
with  the $L^2$ projection operator $P_K$ from $L^2(K)$ onto $Q_K$.
Then a similar argument of Lemma \ref{EnrichedNQ1} completes the proof.
\end{proof}

In the appendixes A and C,   we have proven that $h \lesssim
\|u-u_h\|_h$ when $u\in H^{2}(\Om)$ and that there exist meshes such
that $h^{\mathfrak{s}}\lesssim \|u-u_h\|_h$ holds when $u\in H^{1+\mathfrak{s}}(\Om)$ with
$0<\mathfrak{s}< 1$.  Hence,  the result in Theorem \ref{main}
holds for this class of elements.

\subsection{A new first order nonconforming element}
With the condition from Theorem \ref{main}, a systematic method
obtaining the lower bounds for eigenvalues  is to design
nonconforming finite element spaces with good local approximation
property but not so good  global continuity property.
 To make the idea clearer, we propose a new nonconforming  element that admits lower bounds
 for  eigenvalues.  Let $\cT_h$ be some shape regular decomposition into
 triangles of the polygonal domain $\Om\subset\R^2$.  We define
\begin{equation}\label{EXEX}
 V_h:=\begin{array}[t]{l}\big\{v\in L^2(\Om):
 v|_{K}\in P_{2}(K) \text{ for each }K\in \cT_h,
 \int_f[v]df=0,\\
 \text{ for all internal edges~} f\,,
 \text{ and }\int_f vdf=0 \text{ for all edges  $f$ on  } \Gamma_D
 \big\}\,.
 \end{array}\nn
\end{equation}
Since the conforming quadratic element space on the triangle mesh is a subspace of $V_h$, the usual dual argument proves
$$
\|u-P_hu\|_{L^2(\Omega)}\lesssim h^{2+\mathfrak{s}}|u|_{H^{2+\mathfrak{s}}(\Omega)},
$$
provided that $u\in V\cap H^{2+\mathfrak{s}}(\Om)$ with $0<\mathfrak{s}\leq 1$.  In the appendix A, it is shown that
$h\lesssim \|\nabla_h(u-u_h)\|_{L^2(\Omega)}$, which in fact implies the condition (H4)  for this case.
For the singular case $u\in V\cap H^{1+\mathfrak{s}}(\Om)$, a similar argument of the enriched Crouzeix-Raviart
 element   is able to show the condition (H4).

\section{Conclusion and comments}
In this paper, we propose a systematic method that can produce
lower bounds for  eigenvalues of elliptic operators.  With this
method, to obtain  lower bounds is to design nonconforming
finite element spaces with  enough local degrees of freedom when
compared to the global continuity.  We check that several
nonconforming methods in literature possess  this promising
property.  We also propose some new nonconforming methods with this
feature.   In addition, we  study systematically the saturation
condition for both  conforming and nonconforming finite element
methods.

Certainly,  there are many other nonconforming finite elements which are not analyzed herein.
Let mention several more elements and give some short comments on applications of the theory herein to them.
The first one is  the nonconforming rotated $Q_1$ element from \cite{Rannacher92}. For this element, discrete
 eigenvalues are smaller than exact ones when eigenfunctions are singular, see more details from  \cite{YangZhangLin10}.
 The same comments applies for the Crouzeix--Raviart element of \cite{CroRav73}, see more details from \cite{AD04,YangLinZhang09}.
The last one is the Adini element \cite{Lascaux85,ShiWang10} for fourth order problems.
  For this element,  by an  expansion result of \cite[Lemma]{HuShi2012} and a  similar
identity like that of  Lemma 4.1 therein, a similar argument for the Wilson element is able to show that
discrete eigenvalues are smaller than exact ones provided that eigenfunctions $u\in H^4(\Omega)$.

\appendix

\section{The saturation condition}
In the following two sections, we shall prove, for some cases, the
saturation condition which is used in Theorem \ref{main}. The error
basically consists of two parts: approximation errors and the
consistency errors. In this section, we analyze the case where
approximation errors are dominant and the case where consistency errors are dominant; in the 
appendix B, we give some comments for the case  where
eigenfunctions are singular.

\subsection{The saturation condition  where  approximation error are dominant}  Let  $u\in V\cap H^m(\Om)$ be
eigenfunctions of some $2m$-th order elliptic operator.  Let $V_h$ be
some $k$-th order conforming or nonconforming  approximation spaces
to $H^m(\Om)$ over the mesh $\cT_h$ in the following sense:
\begin{equation}
\frac{\sup\limits_{0\not= v\in H^{m+k}(\Om)\cap V}\inf\limits_{v_h\in
V_h}\|D_h^{m}(v-v_h)\|_{L^2(\Om)}}{|v|_{H^{m+k}}}\lesssim h^k\text{ for
some positive integer } k.
\end{equation}
Then the following  condition is sufficient for the saturation
condition:
   \begin{enumerate}
    \item[H5] At least one fixed component of $D_h^{m+k}v_h$ vanishes for
               all $v_h\in V_h$  while  the $L^2$ norm of  the same component of $D^{m+k}u$
                is  nonzero.
  \end{enumerate}

Recall that $D^{\ell}v$ denote the $\ell$-th order tensor of all
$\ell$-th order derivatives of $v$, for instance, $\ell=1$ the
gradient,  and $\ell=2$ the Hessian matrix, and that $D_h^{\ell}$ are the
piecewise  counterparts of $D^{\ell}$ defined element by element.

In order to  achieve the desired result, we shall use the operator defined in \eqref{QuasiInterpolation}.  For readers'
 convenience, we recall its definition.
Given any  element $K$,  define $J_{m+k, K}v\in P_{m+k}(K)$ by
\begin{equation}\label{QuasiInterpolation1}
\int_{K}D^{\ell}J_{m+k,K}v dxdy=\int_K D^{\ell}v dxdy, \ell=0, 1, \cdots, m+k,
\end{equation}
for any $v\in H^{m+k}(K)$.  Note that the operator $J_{m+k,K}$ is
well-defined. Since $\int_K D^\ell(v-J_{m+k,K}v)dxdy=0$, $\ell=0, \cdots, m+k$,
\begin{equation}\label{QuasiInterpolation4}
\|D^{\ell_1}(v-J_{m+k,K}v)\|_{L^2(K)}\leq Ch_K^{\ell_2-\ell_1}\|D^{\ell_2}(v-J_{m+k,K}v)\|_{L^2(K)}\text{ for any } 0\leq\ell_1\leq\ell_2\leq m+k.
\end{equation}
 Finally,  define the global operator $J_{m+k}$ by
\begin{equation}\label{QuasiInterpolation2}
J_{m+k}|_K=J_{m+k,K} \text{ for any }K\in\cT_h.
\end{equation}
 It follows from the very definition of $J_{m+k,K}$ in
\eqref{QuasiInterpolation1} that
\begin{equation}
D^{m+k}_hJ_{m+k} v=\Pi^0D^{m+k}v,
\end{equation}
where $\Pi^0$ is  the $L^2$ piecewise constant projection operator with respect to $\cT_h$, which is defined in subsection \ref{sub5.1}.   Since piecewise constant functions are dense
 in the space $L^2(\Om)$,
 \begin{equation}\label{QuasiInterpolation3}
 \|D_h^{m+k}(v-J_{m+k} v)\|_{L^2(\Om)}\rightarrow 0 \text{ when }h\rightarrow 0.
 \end{equation}

\begin{theorem}  Under  the condition H5, there holds the following
saturation condition:
\begin{equation}
h^{k}\lesssim \|D_h^m(u-u_h)\|_{L^2(\Om)}.
\end{equation}
\end{theorem}
\begin{proof}  By the condition H5,  we let $\mathfrak{N}$ denote the multi-index set such
that $|\kappa|=m+k$ for any $\kappa\in \mathfrak{N}$ and that
\begin{equation}
\frac{\pa^{\kappa}v_h|_K}{\pa x^{\kappa}}\equiv 0\text{ for any }
K\in\cT_h \text{ and }v_h\in V_h \text{ while }
\|\frac{\pa^{\kappa}u}{\pa x^{\kappa}}\|_{L^2(\Om)}\not=0.
\end{equation}
Let $J_{m+k}$ be defined as in \eqref{QuasiInterpolation1} and \eqref{QuasiInterpolation2}.
 It follows from the triangle inequality and the piecewise
inverse estimate that
\begin{equation}
\begin{split}
&\sum\limits_{\kappa\in \mathfrak{N}} \|\frac{\pa^{\kappa}u}{\pa
x^{\kappa}}\|_{L^2(\Om)}^2=\sum\limits_{\kappa\in
\mathfrak{N}}\sum\limits_{K\in \cT_h}
\|\frac{\pa^{\kappa}(u-u_h)}{\pa x^{\kappa}}\|_{L^2(K)}^2\\[0.3ex]
&\quad \leq 2 \sum\limits_{\kappa\in \mathfrak{N}}\sum\limits_{K\in
\cT_h} \big(\|\frac{\pa^{\kappa}(u-J_{m+k}u)}{\pa
x^{\kappa}}\|_{L^2(K)}^2+ \|\frac{\pa^{\kappa}(J_{m+k}u-u_h)}{\pa
x^{\kappa}}\|_{L^2(K)}^2\big)\\[0,3ex]
&\quad\lesssim
\|D_h^{m+k}(u-J_{m+k}u)\|^2_{L^2(\Om)}+h^{-2k}\|D_h^{m}(J_{m+k}u-u_h)\|_{L^2(\Om)}^2.
\end{split}
\end{equation}
The estimate of \eqref{QuasiInterpolation4} and the triangle
inequality lead to
\begin{equation}
\begin{split}
&\sum\limits_{\kappa\in \mathfrak{N}} \|\frac{\pa^{\kappa}u}{\pa
x^{\kappa}}\|_{L^2(\Om)}^2\lesssim
\|D_h^{m+k}(u-J_{m+k}u)\|^2_{L^2(\Om)}+h^{-2k}\|D_h^{m}(u-u_h)\|_{L^2(\Om)}^2.
\end{split}
\end{equation}
Finally it follows from  \eqref{QuasiInterpolation3} that
\begin{equation}
h^{2k}\sum\limits_{\kappa\in \mathfrak{N}}
\|\frac{\pa^{\kappa}u}{\pa x^{\kappa}}\|_{L^2(\Om)}^2\lesssim
\|D_h^{m}(u-u_h)\|_{L^2(\Om)}^2
\end{equation}
when the meshsize is small enough, which completes the proof.
\end{proof}

\begin{remark}  Under  the condition H5, a similar argument can prove
the following general saturation conditions:
\begin{equation}
h^{k+m-\ell}\lesssim \|D_h^\ell (u-u_h)\|_{L^2(\Om)}, \ell=0, 1, \cdots, m.
\end{equation}
\end{remark}

Next, we prove the condition H5 for various
 element in  literature.
 \begin{enumerate}
\item The Morley-Wang-Xu element.  Since $D^{m+1}_h v_h\equiv 0$ for all $v_h\in V_h$ for this family of elements and $v\equiv 0$
 if $D^{m+1}v\equiv 0$ for any $v\in V\cap H^{m+1}(\Omega)$, the condition H5 holds.
 \item The enriched Crouzeix-Raviart element. Let $\pa_{12,h}$ denote the piecewise counterpart of the
differential operator $\frac{\pa^2}{\pa x\pa y}$.  We have
$\pa_{12,h}v_h\equiv 0$ for any $v_h\in V_h$.  We only consider the case where
$\Omega=[0,1]^2$ and $u\in H_0^1(\Omega)$. If $ \|\frac{\pa^2 v}{\pa x\pa y}\|_{L^2(\Omega)}$
  vanishes for $v\in V\cap H_0^2(\Omega)$. Then, $v$ should be of the
   form $v(x,y)=f(x)+g(y)$, where $f(x)$ is some function of the variable $x$
   and $g(y)$ is some function of the variable $y$. Now the homogenous boundary
   condition indicates  that $f(x)\equiv C_1$ and $g(y)\equiv C_2$ for
   some  constants $C_1$ and $C_2$, which in turn concludes that $v\equiv 0$. This proves the condition
   $H5$.
\item The  same argument applies to the nonconforming $Q_1$ element, the  enriched nonconforming rotated $Q_1$ element, and the
 conforming $Q_1$ element  in any  dimension.
 \end{enumerate}

\subsection{The saturation condition where  consistency errors are dominant}
In this subsection, we  prove the saturation condition  for the case
where  consistency errors are dominant. As usual it is very
  complicated  to give an abstract estimate for consistency errors in
a unifying way.   Therefore, for ease of presentation,  we shall
only  consider  the new first order nonconforming element proposed in this paper.
However, the idea can be extended to other nonconforming finite
element methods.

\begin{figure}[h]
\begin{center}
\setlength{\unitlength}{0.25cm}
\begin{picture}(26,13.8)(0,0)
\put(10,1.3){\begin{picture}(22,13)(0,0)
\put (0,0){\line(1,0){10}}
\put (0,0){\line(0,1){10}}
\put(10,0){\line(-1,1){10}}
\put (0,0){\line(-1,0){10}}
\put(-10,0){\line(1,1){10}}
\put(-1,7){e}
\put(-3.33,0){\circle*{1}}
\put(-6.67,0){\circle*{1}}
\put(3.33,0){\circle*{1}}
\put(6.67,0){\circle*{1}}
\put(3.33,6.67){\circle*{1}}
\put(6.67,3.33){\circle*{1}}
\put(-3.33,6.67){\circle*{1}}
\put(-6.67,3.33){\circle*{1}}
\put(2.5, 2.5){\circle*{1}}
\put(-2.5, 2.5){\circle*{1}}
\put(-14,0){(-1,0)}
\put(10,0){(1,0)}
\put(-1,11){(0,1)}
\put(0.6, 3.6){$(\frac{1}{4}, \frac{1}{4})$}
\put(-4.6, 3.6){$(\frac{-1}{4}, \frac{1}{4})$}

\end{picture}
}

\end{picture}
\end{center}
\caption{Reference Edge patch and degrees of freedom for $v_e$}\label{figure1}
\end{figure}
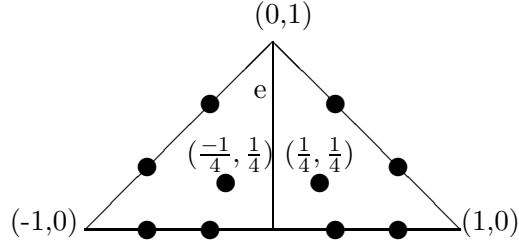

 In order to give lower bounds of consistency errors, given any edge ( boundary and interior) $e$, we construct
 functions $v_e\in V_h $ such that:
 \begin{enumerate}
 \item $v_e$ vanishes on $\Omega\backslash\omega_e$;
 \item $v_e$ vanishes on two Gauss-Legendre points of the other  four edges than $e$ of $\omega_e$;
 \item $v_e$  vanishes at   two interior  points of $\omega_e$, see points $(\frac{1}{4}, \frac{1}{4})$ and   $(\frac{-1}{4}, \frac{1}{4})$  in Figure \ref{figure1} for examples of the reference edge patch;
 \item  $\int_e[v_e]sds=\mathcal{O}(h^2)\not= 0$.
 \end{enumerate}
See Figure \ref{figure1} for the reference edge patch and degrees of freedom for $v_e$. Note that such a function can be found. In fact,
 for the reference edge patch in Figure \ref{figure1}, a  direct calculation shows that there exists  a function $v_e\in V_h$ such that $\int_e[v_e]sds=0.1715\not=0$.

Let $\Pi_e^1$ be the $L^2$ projection from $L^2(e)$ to $P_1(e)$.  Since $\int_e[v_h]ds=0$ for any edge $e$ of $\cT_h$ and $v_h\in V_h$,
 it follows that
\begin{equation}\label{New}
\sum\limits_{K\in\cT_h}\int_{\pa K}\frac{\pa u}{\pa \nu}v_hds=\sum\limits_{e}\int_e\frac{\pa u}{\pa \nu}[v_h]ds
=\sum\limits_{e} \int_e\frac{\partial}{\partial \tau}\big(\Pi_e^1\frac{\pa u}{\pa \nu}\big)[v_h]sds+\sum\limits_{e}\int_e(I-\Pi_e^1)\frac{\pa u}{\pa \nu}[v_h]ds.
\end{equation}
Define
\begin{equation}
v_h=\sum\limits_{e}v_e\frac{\partial}{\partial \tau}\big(\Pi_e^1\frac{\pa u}{\pa \nu}\big).
\end{equation}
Since $\frac{\partial}{\partial \tau}\big(\Pi_e^1\frac{\pa u}{\pa \nu}\big)$ are constants,  definitions of $v_e$ yield
\begin{equation*}
\sum\limits_{e} \int_e\frac{\partial}{\partial \tau}\big(\Pi_e^1\frac{\pa u}{\pa \nu}\big)[v_h]sds
\geq Ch\sum\limits_{e}\|\frac{\partial}{\partial \tau}\big(\Pi_e^1\frac{\pa u}{\pa \nu}\big)\|_{L^2(e)}^2,
\end{equation*}
and
\begin{equation*}
\|\nabla_h v_h\|_{L^2(\Omega)}\leq C h^{-1/2}\bigg( \sum\limits_{e}\|\frac{\partial}{\partial \tau}\big(\Pi_e^1\frac{\pa u}{\pa \nu}\big)\|_{L^2(e)}^2\bigg)^{1/2}.
\end{equation*}
A substitution of these two inequalities into \eqref{New} leads to
\begin{equation}
 \sup\limits_{0\not=v_h\in V_h}\frac{\sum\limits_{K\in\cT_h}\int_{\pa K}\frac{\pa u}{\pa \nu}v_hds}{\|\nabla_hv_h\|_{L^2(\Omega)}}
  \geq C_1h\|\nabla^2u\|_{L^2(\Omega)}-C_2h^{1+s}|u|_{H^{2+s}(\Om)},
\end{equation}
provided that $u\in H^{2+s}(\Omega)$ with $0<s\leq 1$ for some
positive constants $C_1$ and $C_2$. Since
$\|\nabla^2u\|_{L^2(\Omega)}$ can not vanish, this proves the
saturation condition.
\begin{remark} Thanks to two nonconforming bubble functions in each element,  a similar argument is able to
  show a corresponding result for the Wilson element \cite{ShiWang10,Wilson73}.
\end{remark}
  \section{The comment for  the saturation condition of the singular case}\label{secA.6}
   We need the concept of the interpolation space.
  Let $X$, $Y$ be a pair of normed linear spaces.   We shall assume that $Y$ is continuously
embedded in $X$ with $Y\subset X$ and $\|\cdot\|_X\lesssim \|\cdot\|_Y$. For any $t\geq 0$,  we define the $K-$functional
\begin{equation}
K(f,t)=K(f,t,X,Y)=\inf\limits_{g\in Y}\|f-g\|_X+t|g|_Y,
\end{equation}
where $\|\cdot\|_X$ is the norm on $X$ and $|\cdot|_Y$ is a
semi-norm on $Y$. The function $K(f,.)$ is defined  on $\R_+$ and is
monotone and concave (being the pointwise infimum of linear
functions). If $0<\theta<1$ and $1<q\leq\infty$, the interpolation
space
 $(X,Y)_{\theta, q}$ is defined as the set of all functions $f\in X$ such
 that \cite{BerghLofstrom,DeVore1998,DeVoreLorentz1993}
 \begin{equation}
 |f|_{(X,Y)_{\theta, q}}=\left\{\begin{array}{l}(\sum\limits_{k=0}^{\infty}[2^{(s+\epsilon)k\theta}K(f,2^{-k(s+\epsilon)})]^q)^{1/q}, 0<q<\infty,\\
 \sup\limits_{k\geq 0}2^{k(s+\epsilon)\theta}K(f, 2^{-k(s+\epsilon)}), q=\infty,
 \end{array}\right.
 \end{equation}
is finite for some $0<s+\epsilon\leq 1$.
\subsection{An abstract theory}
 We assume that $u\in H^{m+s}(\Omega)$ with $0<s<1$ and $V_h$ is
 some nonconforming or conforming approximation space to the space $H^m(\Om)$.

Then the following conditions imply in some sense the saturation
condition for the singular case:
\begin{enumerate}
\item[H6.]  There exists a piecewise  polynomial space $V_{m+1,h}^c\subset H^{m+1}(\Om)$
 such that $ V_{m+1,h}^c\subset V_{m+1,h/2}^c$ when $\cT_{h/2}$
    is some nested conforming refinement of $\cT_h$;
\item[H7.] There holds the following Berstein inequality
 \begin{equation}\label{Bernstein}
 |v|_{H^{m+s+\epsilon}(\Omega)}\lesssim
 h^{-(s+\epsilon)}|v|_{H^m(\Omega)}\text{ for any }v\in V_{m+1,h}^c;
  \end{equation}
  \item[H8.]  There exists some quasi-interpolation operator $\Pi^c: V_h\rightarrow V_{m+1, h}^c$ such that
  \begin{equation}
  \|D^m(u-\Pi^c u_h)\|_{L^2(\Om)}\lesssim h^{s+\epsilon}
  \end{equation}
  provided that $\|D_h^m(u-u_h)\|_{L^2(\Om)}\lesssim h^{s+\epsilon}$
  with $\epsilon>0$ and $s+\epsilon\leq 1$.
\end{enumerate}

\begin{theorem} Suppose the eigenfunction $u\in H^{m+s}(\Om)$
with $0<s<1$. Under conditions H6--H8, there exist meshes such that
the following saturation condition holds
\begin{equation}
h^s\lesssim \|D_h^m(u-u_h)\|_{L^2(\Om)}.
\end{equation}
\end{theorem}

\begin{proof}
We assume that the saturation condition $h^{s}\lesssim
\|D_h^m(u-u_h)\|_{L^2(\Omega)}$ does not hold
  for any mesh $\cT_h$ with the meshsize $h$.  In other word, we have
  \begin{equation}\label{A.34}
  \|D_h^m(u-u_h)\|_{L^2(\Omega)}\lesssim h^{s+\epsilon},
  \end{equation}
  for some $\epsilon>0$.  In the following, we assume that $s+\epsilon\leq 1$.
By the condition H8,  we have
 \begin{equation}\label{Jackson}
 \begin{split}
 \inf\limits_{v\in V_{m+1,h}^c}\|D^m(u-v)\|_{L^2(\Omega)}&\lesssim \|D_h^m(u-\Pi^c u_h)\|_{L^2(\Omega)}\lesssim h^{s+\epsilon}.
\end{split}
 \end{equation}
 Take $X=H^m(\Omega)$ and $Y=H^{m+s+\epsilon}(\Omega)$.  The inequality \eqref{Jackson} is the
usual Jackson inequality and the inequality \eqref{Bernstein} is the
Berstein inequality  in the  context of the approximation theory
\cite{DeVoreLorentz1993,DeVore1998}.  We can follow the idea of
\cite[Theorem 5.1 , Chapter 7]{DeVoreLorentz1993} to estimate terms
like $K(u, 2^{-{\ell}(s+\epsilon)})$ for any positive integer $\ell$. In fact,
let $\varphi_k\in V_{m+1, 2^{-k(s+\epsilon)}}^c$ be the best approximation to
$u$ in the sense that
$\|D^m(u-\varphi_k)\|_{L^2(\Omega)}=\inf\limits_{v\in
V_{m+1,2^{-k(s+\epsilon)}}^c}\|D^m(u-v)\|_{L^2(\Omega)}$, $k\geq 1$. Let
$\psi_k=\varphi_k-\varphi_{k-1}$, $k=1, 2, \cdots,$ where
$\psi_0=\varphi_0=0$.  Then we have
\begin{equation}\label{B8}
\|D^m\psi_k\|_{L^2(\Omega)}\leq
\|D^m(u-\varphi_k)\|_{L^2(\Omega)}+\|D^m(u-\varphi_{k-1})\|_{L^2(\Omega)}\lesssim
2^{-k(s+\epsilon)}.
\end{equation}
Since $\varphi_{\ell}=\sum\limits_{k=0}^{\ell}\psi_k$ and
$|\psi_0|_{H^{m+s+\epsilon}(\Omega)}=0$, it follows from
\eqref{Jackson},  \eqref{Bernstein} and \eqref{B8} that
 \begin{equation}
 \begin{split}
  K(u, 2^{-(s+\epsilon)\ell})&\leq
 \|u-\varphi_{\ell}\|_{H^{m}(\Omega)}+2^{-(s+\epsilon)\ell}|\varphi_{\ell}|_{H^{m+s+\epsilon}}\\[0.3ex]
 &\lesssim
 2^{-(s+\epsilon)\ell}+2^{-(s+\epsilon)\ell}\sum\limits_{k=1}^{\ell}2^{k(s+\epsilon)^2}\|\psi_k\|_{H^m(\Omega)}\\[0.3ex]
 &\lesssim \ell 2^{-(s+\epsilon)\ell}.
 \end{split}
  \end{equation}
  \begin{equation}
  \begin{split}
 |u|_{(H^{m}(\Omega), H^{m+s+\epsilon}(\Omega))_{\theta, 2}}
& =\bigg(\sum\limits_{k=0}^{\infty}\big[2^{k(s+\epsilon)\theta}K(u,2^{-k(s+\epsilon)})\big]^2\bigg)^{1/2}\lesssim \bigg(\sum\limits_{k=0}^{\infty}\big[k2^{k(s+\epsilon)(\theta-1)}\big]^2\bigg)^{1/2}.
 \end{split}
 \end{equation}
 Let $\theta=1-\epsilon_0$ with $\epsilon_0>0$ such that $\epsilon-(s+\epsilon)\epsilon_0>0$.  This leads to
 \begin{equation}
 |u|_{(H^{m}(\Omega), H^{m+s+\epsilon}(\Omega))_{\theta, 2}}
\lesssim \bigg(\sum\limits_{k=0}^{\infty}\big[k2^{-k(s+\epsilon)\epsilon_0}\big]^2\bigg)^{1/2}<\infty.
 \end{equation}
 This proves that $u\in H^{m+(1-\epsilon_0)(s+\epsilon)}(\Omega)$
 which is a proper subspace of $H^{m+s}(\Omega)$ since $\epsilon-(s+\epsilon)\epsilon_0>0$,
 which contradicts with the fact that  we only have the regularity $u\in H^{m+s}(\Omega)$.
\end{proof}

\subsection{Proofs for H6--H8}
It follows from  \cite{Da01} that there exist piecewise polynomial spaces $V_{m+1,h}^c$ with nodal basis over $\cT_h$ such
 that $V_{m+1,h}^c$ are nested and  conforming in the sense that $V_{m+1,h}^c\subset V_{m+1,h/2}^c\subset H^m(\Om)$
 for any $1\leq n$ and $m\leq n$.

 This result  actually proves the conditions H6 and H7. The proof of H8 needs the interpolation of $V_h$  into the conforming finite element space. To this end, we introduce the projection average interpolation operator of \cite{Brenner2003,ShiWang10}.

Let $V_{m+1,h}^c$ be a conforming finite element
space defined by $(K,P_K^c,D^c_K)$, where $D_T^c$ is the vector functional and the components of $D^c_K$ are defined as follows:
for any $v\in C^{\kappa}(K)$
\begin{equation*}
d_{i,K}(v):=\left\{\begin{aligned}
&D_{i,K}v(a_{i,K}) & 1\leq i\leq k_1,\\
&\frac{1}{|F_{i,K}|}\int_{F_{i,K}}D_{i,K}v\diff s & k_1< i\leq k_2,\\
&\frac{1}{|K|}\int_KD_{i,K}v\diff x & k_2< i\leq L,\\
\end{aligned}\right.\leqno(*)
\end{equation*}
where $a_{i,K}$ are points in $K$, $F_{i,K}$ are non zero-dimensional faces of $K$.
$\kappa:=\max\limits_{1\leq i\leq L}k(i)$ where $k(i)$ orders of derivatives used in  degrees of freedom  $D_{i,K}:=\sum\limits_{|\alpha|=k(i)}\eta_{i,\alpha, K}\partial^{\alpha},1\leq i\leq L$,
$\eta_{i,\alpha, K}$  are constants which depend on $i, \alpha,  \text{ and } K$.

Let $\omega(a)$ denote the union of elements that share
point $a$ and $\omega(F)$ the union of elements having in common the face $F$. Let $N(a)$ and $N(F)$ denote the number of elements in $\omega(a)$ and $\omega(F)$, respectively. For any $v\in V_h$, define the projection average interpolation operator $\Pi^c: V_h\rightarrow V_{m+1,h}^c$ by
\begin{enumerate}
        \item for $1\leq i\leq k_1$, if $a_{i,K}\in\partial\Omega$ and $d_{i,K}(\phi)=0$ for any $\phi\in C^{\kappa}(\overline{\Omega})\cap V$,
        then $d_{i,K}(\Pi^c v|_K):=0$; otherwise
        \begin{equation*}
        d_{i,K}(\Pi^c v|_K):=\frac{1}{N(a_{i,K})}\sum_{K'\in\omega(a_{i,K})}D_{i,K}(v|_{K'})(a_{i,K});
        \end{equation*}
        \item for $k_1< i\leq k_2$, if $F_{i,K}\subset\partial\Omega$ and $d_{i,K}(\phi)=0$ for any $\phi\in C^{\kappa}(\overline{\Omega})\cap V$,
        then $d_{i,K}(\Pi^c v|_K):=0$; otherwise
        \begin{equation*}
        d_{i,K}(\Pi^c v|_K):=\frac{1}{N(F_{i,K})}\sum_{K'\in\omega(F_{i,K})}\frac{1}{|F_{i,K}|}\int_{F_{i,K}}D_{i,K}(v|_{K'})(a_{i,K})\diff s;
        \end{equation*}
        \item for $k_2< i\leq L$
        \begin{equation*}
        d_{i, K}(\Pi^c v|_K):=\frac{1}{|K|}\int_{K}D_{i,K}(v|_{K})\diff x.
        \end{equation*}
\end{enumerate}
\begin{lemma} For all nonconforming element spaces under consideration,  there
exists $r\in \mathbb{N}, r\geqslant m$ such that $V_h|_K\subset P_r(K)\subset P^c_K$. Then, for
$m<k\leqslant min\{r+1,2m\}$, $0\leqslant l\leqslant m$,
$\alpha=(\alpha_1,\cdots,\alpha_n)$, it holds that
\begin{equation*}
\begin{aligned}
\|D_h^m(v_h-\Pi^c v_h)\|^2_{L^2(\Omega)}&\lesssim
\sum_{K\in\mathcal{T}_h}\left(\sum^{k-1}_{j=m}h_K^{2(j-m)+1}\sum_{F\subset\partial
K/\partial\Omega}\sum_{|\alpha|=j}\|[\partial^{\alpha}v_h]\|^2_{0,F}\right.\\
&\left. +h_K\sum_{F\subset\partial
K\cap\partial\Omega}\sum_{|\alpha|=m,\alpha_1<m}\|\frac{\partial^{|\alpha|}v_h}{\partial
\nu_F^{\alpha_1}\partial\tau_{F,2}^{\alpha_2}\cdots\partial\tau_{F,n}^{\alpha_n}}\|^2_{0,F}\right),
\end{aligned}
\end{equation*}
where $\tau_{F,2},\cdots,\tau_{F,n}$ are $n-1$ orthonormal tangent vectors of $F$.
\end{lemma}
\begin{proof}
Since $V_h|_K\subset P_r(K)\subset P^c_K$,
a slight modification of the argument in \cite[Lemma 5.6.4]{ShiWang10} can prove the desired result;
 see also \cite{Brenner2003} for the proof of  the nonconforming linear element with $m=1$.
\end{proof}

The remaining proof is based on bubble function techniques, see \cite{CarstensenHu07}  for a posteriori error analysis of second order problems,
 see \cite{Gudi(2010),HuShi(2009)}  for a posteriori error analysis of fourth order problems.
 Let $v_h=u_h$ in the above lemma. Such an analysis leads to
  \begin{equation}
  \|D_h^m(\Pi u_h-u_h)\|_{L^2(\Om)}\lesssim \|D_h^m(u-u_h)\|_{L^2(\Om)}\lesssim h^{s+\epsilon}.
  \end{equation}


\begin{thebibliography}{ccm}

\bibitem{AD03} M.~G. Armentano and R. G. Duran.
  Mass-lumping or not mass-lumping for eigenvalue problem.
Numer. Methods PDEs, 19(2003), pp: 653--664.

\bibitem{AD04} M.~G.~ Armentano and R.~G. Duran.
  Asymptotic lower bounds for eigenvalues by nonconforming finite
 element methods.   ETNA,   17(2004), pp: 93--101.


\bibitem{BabuskaKellog1979}I.  Babu$\breve{s}$ka, R. B. Kellogg and J. Pitk$\ddot{a}$ranta.  Direct and inverse error estimates
  for finite elements with mesh refinements.  Numer. Math.,  33(1979), pp: 447--471.


\bibitem{BabuskaStrouboulis2000} I. Babu$\breve{s}$ka  and T. Strouboulis.  The finite element method and its reliability. 
 Oxford Science Publications, 2000.


\bibitem{BA91}I.~Babu$\breve{s}$ka  and J.~E.~Osborn.
   Eigenvalue Problems, in Handbook of Numerical Analysis,V.II:
   Finite Element Methods(Part I), Edited by P.G.Ciarlet and J.L.Lions,
  1991, Elsevier.

\bibitem{BerghLofstrom} J.~Bergh and J. L\"{o}fstr\"{o}m.
Interpolation Spaces:  An Introduction.  Springer-Verlag Berlin
Heidelberg, 1976.

\bibitem{Boffi2010}D. Boffi.   Finite element approximation
of eigenvalue problems.   Acta Numerica (2010), pp:  1--120.



\bibitem{Brenner2003}S.~ Brenner. Poincar\'{e}-Friedrichs inequality for piecewise $H^1$ functions.
SIAM  J. Numer. Anal.,  41(2003),  pp:  306--324.


\bibitem{BrennerScott} S.~C.~Brenner and L.~ R.~ Scott.  The
mathematical theorey of finite element methods.  Springer-Verlag,
1996.


\bibitem{CarstensenGedicke11}C. Carstensen and J. Gedicke.  An oscillation-free adaptive FEM for symmetric eigenvalue
   problem.  Numer. Math., 118(2011), pp: 401--427.

\bibitem{CarstensenHu07}C. Carstensen and J. Hu.  A unifying theory of a
posteriori error control for  nonconforming finite element methods.
Numer. Math.,  107(2007), pp: 473--502.



\bibitem{ChenLi1994} H. S. Chen and B. Li.    Superconvergence analysis and error expansion for the Wilson nonconforming finite element.
Numer. Math.,  69(1994), pp: 125--140.


\bibitem{CroRav73} M. Crouzeix and P.-A. Raviart.  Conforming and nonconforming finite element methods
        for solving the stationary Stokes equations. RAIRO Anal. Num\'er.,  7 (1973), pp: 33--76.

\bibitem{Da01}O.Davydov.   Stable local bases for multivariate spline
space. J. Appr. Theory,  111(2001), pp: 267--297.


\bibitem{DeVore1998} R. A. DeVore.  Nonlinear approximation. Acta Numerica,  (1998), pp: 51--150.


\bibitem{DeVoreLorentz1993} R.  A.  DeVore and G. G. Lorentz.
Constructive approximation.  Springer-Verlag, 1993.


\bibitem{Forsythe54} G.~E.~Forsythe.  Asymptotic lower bounds for
the frequencies of certain polygonal membranes.  Pacific J. Math.,
4(1954), pp: 467--480.

\bibitem{Forsythe55} G.~E.~Forsythe, Asymptotic lower bounds for
the fundamental frequency of convex membranes.  Pacific J. Math.
4(1954), pp: 691--702.

\bibitem{Gudi(2010)}
T. Gudi. A new error analysis for discontinuous finite
element methods for linear elliptic problems.  Math. Comp., (2010), pp: 2169--2189.



\bibitem{HH04}J. Hu,  Y. Q. Huang and H.~M.~ Shen.  The lower approximation
 of eigenvalue by lumped mass finite element methods. J.~Comp.~Math.,
  22(2004),  pp: 545--556.

\bibitem{HuShi(2009)}
J. Hu and Z. C. Shi. A new a posteriori error estimate for the Morley element. Numer. Math., 112 (2009), pp: 25--40.


\bibitem{HuShi2012}J. Hu and Z. C. Shi. The lower bound of the error estimate in
the $L^2$ norm for the Adini element of the biharmonic equation.  arXiv:1211.4677 [math.NA], 2012.



\bibitem{KrizekRoosChen2011}M. K\v{r}\'{i}\v{z}ek, H. Roos and W. Chen.  Two-sided bounds of the discretizations
  error for finite elements.  ESAIM: M2AN,  45(2011),  pp: 915--924.


\bibitem{Li08} Y.~A.~Li.  Lower approximation of eigenvalue by the nonconforming  finite element method.
Math. Numer. Sin.,  30(2008), pp: 195--200.



\bibitem{Lascaux85}P. Lascaux and P. Lesaint. Some nonconforming finite elements for
the plate bending problem.  RAIRO Anal. Numer.,  1 (1975), pp: 9--53.


\bibitem{LinHuangLi08}Q.~Lin, H.~T.~Huang and Z.~C.~Li.  New
expansions  of numerical eigenvalues  for $-\triangle u=\lambda\rho
u$ by nonconforming elements.   Math. Comp.,  77(2008), pp: 2061--2084.


\bibitem{LinHuangLi09}Q.~Lin, H.~T.~Huang and Z.~C.~Li.  New
expansions of numerical eigenvalues by Wilson's element. J. Comp.
Appl. Math.,  225(2009), pp: 213--226.



\bibitem{Ltz04} Q.~Lin, L.~Tobiska and A.~Zhou.
 On the superconvergence  of nonconforming low order finite
elements applied to the Poisson equation.  IMA ~J.~Numer.~Anal.,
25(2005), pp: 160--181.


\bibitem{LinLin06} Q. Lin and J.~Lin.  Finite element methods:
Accuracy and improvements.  Science Press, Beijing, 2006.

\bibitem{LiuYan05}H.~P.~ Liu  and  N.~N.~Yan.  Four finite element solutions and comparison of
problem for the poisson equation eigenvalue. Chinese J. Numer. Meth.
Comput. Appl.,  2(2005), pp: 81--91.


\bibitem{Morley68} L. S. D. Morley.   The triangular equilibrium element
     in the solutions of plate bending problem.  Aero.Quart., 19(1968),
     pp: 149--169.


\bibitem{Rannacher79}R.~Rannacher.  Nonconforming finite element methods for eigenvalue
problems in linear plate theory.   Numer. Math.,  33(1979), pp: 23--42.

\bibitem{Rannacher92} R.~Rannacher and S.~Turek.  Simple nonconforming quadrilateral Stokes
element.   Numer. Methods PDEs,  8(1992), pp: 97--111.


\bibitem{Shi1986} Z. C. Shi.  A remark on the optimal order of convergence of Wilson nonconforming element. 
 Math. Numer. Sin.,  8(1986),   pp. 159--163(in Chinese).

\bibitem{ShiWang10}Z.~C.~Shi and M.~ Wang. The finite element
method. Science Press, Beijing, 2010 (in Chinese).


\bibitem{StrangeFix73}G. Strang and G. Fix.  An analysis of the finite element
method.  Prentice-Hall, 1973.




\bibitem{WX06}M.  Wang and J. C. Xu.
Minimal finite-element spaces for 2m-th order partial differential
equations in $R^n$.   Math. Comp., 82(2013), pp: 25--43.


\bibitem{Weinberger56}H.~F.~Weiberger.  Upper and lower bounds by finite difference methods. 
  Comm. Pure Appl. Math.,  9(1956), pp: 613--623.

\bibitem{Weinberger58}H.~F.~Weiberger. Lower bounds for higher eigenvalues by finite difference
methods.  Pacific J. Math., 8(1958), pp: 339--368.


\bibitem{Widlund1977}O. Widlund.  On best error bounds for approximation by piecewise polynomial functions. Numer. Math.,  27(1977),
  pp: 327--338.

\bibitem{Wilson73}E.~L.~Wilson, R.~L.~Taylor, W.~P.~Doherty, J. Ghaboussi.
Incompatible displacement methods, in: S.J. Fenves (Ed.), Numerical
and Computer Methods in Structural Mechanics, Academic Press, New
York, 1973, pp: 43--57.


\bibitem{Yang2000}Y.  D. Yang.  A posteriori error estimates in Adini finite element for eigenvalue problems.
 J.  Comp.  Math.,  18(2000),  pp: 413--418.

\bibitem{YangLinZhang09}Y.~D.~Yang, F.~B.~Lin, and Z.~M.~Zhang.  N-simplex  Crouzeix-Raviart element for
second order elliptic/eigenvalue problems.  Inter. J. Numer. Anal.
Model.,  6(2009), pp: 615--626.

\bibitem{YangZhangLin10}Y.~D.~Yang,  Z.~M.~Zhang and F.~B.~Lin.  Eigenvalue approximation from
below using nonforming  finite elements.  Science in  China:
Mathematics,  53(2010), pp: 137--150.


\bibitem{YangLinBiLi} Y.~D.~ Yang, Q.~ Lin, H.~ Bi, and   Q.~ Li.  Lower eigenvalues approximation by Morley
elements.  Adv.Comput. Math., 36(2012), pp: 443--450.



\bibitem{ZhangYangChen07}Z. Zhang, Y. Yang and  Z. Chen, Eigenvalue approximation from below
by Wilson's elements.  Chinese J. Num. Math. Appl.,  29 (2007), pp: 81--84.


\bibitem{Zienkiewicz67}O.~C.~Zienkiewicz and Y.~K.~ Cheung.  The Finite Element Method in
Structrural and Continuum Mechanics.  New York: McGraw-Hill, 1967.



\end{thebibliography}
\end{document}